\documentclass[12pt]{amsart}

\usepackage[a4paper,margin=32mm]{geometry}
\usepackage{amsmath,amssymb,amsthm,mathtools}
\usepackage[colorlinks=true,linkcolor=blue,citecolor=blue,urlcolor=blue]{hyperref}
\usepackage[nameinlink,capitalize]{cleveref}

\newtheorem{theorem}{Theorem}[section]
\newtheorem{proposition}[theorem]{Proposition}
\newtheorem{corollary}[theorem]{Corollary}
\newtheorem{lemma}[theorem]{Lemma}
\newtheorem{conjecture}[theorem]{Conjecture}
\newtheorem{problem}[theorem]{Problem}
\theoremstyle{definition}
\newtheorem{definition}[theorem]{Definition}
\newtheorem{example}[theorem]{Example}
\theoremstyle{remark}
\newtheorem{remark}[theorem]{Remark}

\newcommand{\kk}{\mathbf{k}}
\newcommand{\im}{\operatorname{im}}
\newcommand{\Hilb}{\operatorname{Hilb}}
\newcommand{\Span}{\operatorname{span}}
\newcommand{\Match}{\operatorname{Match}}
\newcommand{\rank}{\operatorname{rank}}

\newcommand\arxiv[1]{\href{https://arxiv.org/abs/#1}{\texttt{arXiv:#1}}}

\title[Matching and Dunkl matching algebras]{Matching algebras and their Dunkl subalgebras}

\author[D. Grinberg]{Darij Grinberg}
\address{Department of Mathematics, Drexel University, Philadelphia, PA, USA}
\email{darijgrinberg@gmail.com}

\author[A. Kirillov]{Anatol Kirillov}
\address{Yanqi Lake Beijing Institute of Mathematical Sciences and Applications, Huairou District, Beijing, China}
\email{kirillov@bimsa.cn}

\author[B. Shapiro]{Boris Shapiro}
\address{Department of Mathematics, Stockholm University, Stockholm, Sweden and Department of MCS, Guangdong Technion
Israel Institute of Technology,  Shantou, Guangdong province, China}
\email{shapiro@math.su.se}

\date{28 October 2026}

\subjclass[2020]{Primary 05E10, 05C70; Secondary 13A02, 05B35}
\keywords{matching, Dunkl element, Specht polynomial, graph connectivity,
Specht module, Lefschetz property, zonotopal algebra, Tutte polynomial}

\begin{document}
\raggedbottom

\begin{abstract}
For a finite simple graph $G = (V, E)$, we consider the quotient algebra
$\mathcal M(G)$ of the polynomial ring in the variables $u_e$ (for $e \in E$)
by the ideal generated by all products $u_e u_f$ for non-disjoint edges
$e$ and $f$ (this includes all squares $u_e^2$).
This quotient is called the matching algebra of $G$, since it has a basis
indexed by the matchings of $G$.
In this quotient, we define a subalgebra $\mathcal D(G)$ generated by
the signed incidence sums $\theta_v=\sum_{e=(v,w)}u_e-\sum_{e=(w,v)}u_e$
for all $v \in V$ (where all edges of $G$ are oriented arbitrarily);
we call this the Dunkl matching algebra.

We show that, as a graded vector space, $\mathcal D(G)$ is dual to the span of all
polynomials $p_M = \prod_{(i,j) \in M} (x_i - x_j)$,
where $M$ ranges over all matchings of $G$.
For the complete graph $K_n$, the latter span is a direct sum of
two-row Specht modules (one in each degree);
thus its Hilbert series is that of the Catalan triangle, and, in
characteristic zero, the Dunkl matching algebra can be presented by
linear and quadratic relations.
For arbitrary graphs, we formulate the saturation
problem of deciding when the selected matching Specht generators $p_M$
in a given degree $k$ span the full two-row Specht module $S^{(n-k,k)}$.
We show that saturation in degree $k$ forces
$k$-connectivity, that saturation in degree $2$ is equivalent to
$2$-connectivity, and that $k$-linked graphs are saturated in degree $k$.
We also prove saturation in every possible degree whenever the complement of
$G$ is a matching.  We show that the Dunkl matching algebra equals
the full matching algebra exactly for forests, and give an explicit Hilbert
series formula for unicyclic graphs.  Finally, we compare with a related
Postnikov--Shapiro quotient: depending on the ambient algebra, its Hilbert
series is either the central zonotopal specialization of the Tutte polynomial
or a regraded all-terminal reliability polynomial.
\end{abstract}

\maketitle

\section{Introduction}

Let $G=(V,E)$ be a finite simple graph.  The matching generating polynomial
of $G$ records the numbers of sets of pairwise vertex-disjoint edges.  In this
note we lift this elementary invariant to a graded algebra
\[
  \mathcal M(G)=\kk[u_e:e\in E]/(u_eu_f:e\cap f\ne\varnothing).
\]
In this algebra, we define the signed incidence sums
\[
  \theta_v=\sum_{e=(v,w)}u_e-\sum_{e=(w,v)}u_e
  \qquad \text{ for all } v \in V,
\]
where an orientation of $G$ has been chosen.  We denote by $\mathcal D(G)$
the subalgebra generated by the $\theta_v$ and call it the \emph{Dunkl
matching algebra}.
The goal of the present note is to understand this algebra
$\mathcal D(G)$, in particular its Hilbert series.

The homogeneous parts $\mathcal D(G)_k$ of $\mathcal D(G)$
stand in a natural duality to certain subspaces of the polynomial ring
$\kk[x_v \mid  v\in V]$.
Namely, for each matching $M$, define the polynomial
\[
  p_M = p_M(x)=\prod_{(i,j) \in M}(x_i-x_j) \in \kk[x_v \mid  v\in V].
\]
For any $k \geq 0$, we define the vector subspace
\[
  P_k(G):=\Span_\kk\{p_M:M\text{ is a $k$-matching of $G$}\}
  \qquad \text{ of } \kk[x_v \mid  v\in V].
\]
The main results we shall show are the following:

\begin{itemize}
\item
For each $k \geq 0$, we have
\[
  \dim \mathcal D(G)_k=\dim P_k(G);
\]
see \cref{thm:product-realization}.
More precisely, the proof gives a canonical perfect pairing between
$\mathcal D(G)_k$ and $P_k(G)$
(see \cref{rmk:duality});
thus, $P_k(G)$ becomes a dual polynomial model for
$\mathcal D(G)_k$.
Consequently, the Hilbert function of $\mathcal D(G)$ records the ranks of the
matching-restricted products $\prod_{(i,j)\in M}(x_i-x_j)$.

\item
For the complete graph $K_n$, the dual polynomial model becomes classical
representation theory: for $0\leq k\leq\lfloor n/2\rfloor$,
\[
  P_k(K_n)\cong S^{(n-k,k)}.
\]
Here $S^{(n-k,k)}$ is the Specht module of shape $(n-k,k)$
(and is irreducible in characteristic zero).
Thus, the Hilbert series of $\mathcal D(K_n)$ is the
generating function of the $n$-th row of the Catalan triangle; see
\cref{thm:complete-specht}.  In characteristic zero, we also obtain the
presentation
\[
 \mathcal D(K_n)\cong
 \kk[z_1,\ldots,z_n]/
 (z_1+\cdots+z_n,z_1^2,\ldots,z_n^2);
\]
see \cref{thm:presentation}.

\item
Motivated by the complete-graph case, we say that a graph $G$
is \emph{$k$-saturated} when $P_k(G)=P_k(K_n)$.  We prove the implications
\[
 k\text{-linked}
 \ \Longrightarrow\
 k\text{-saturated}
 \ \Longrightarrow\
 k\text{-connected};
\]
see \cref{prop:k-linked,prop:k-connected}.  In degree $2$ the second
implication is an equivalence:
\[
 P_2(G)=P_2(K_n)
 \quad\Longleftrightarrow\quad
 G\text{ is $2$-connected}
\]
for $n\ge4$; see \cref{thm:degree-two-saturation}.  We also show that if
$\overline G$ is a matching, then $G$ is saturated in every possible degree;
see \cref{thm:matching-complement}.

\item
The matching generating polynomial and the Dunkl Hilbert series of $G$
(that is, the Hilbert series of $\mathcal D(G)$) are
incomparable invariants: neither determines the other.  Indeed, $P_3$ and
$K_3$ have the same Dunkl Hilbert series $1+2t$ but different matching
generating polynomials $1+2t$ and $1+3t$, respectively; in the other
direction, \cref{prop:matching-polynomial-counterexample} gives graphs with
the same matching generating polynomial but different Dunkl Hilbert series.
We also obtain explicit positive formulas in sparse cases.  Namely,
\[
  \mathcal D(G)=\mathcal M(G)
  \quad\Longleftrightarrow\quad
  G\text{ is a forest};
\]
see \cref{prop:forest-full}.  For a unicyclic graph with unique cycle $C$ and
$F=G-V(C)$, we obtain a short exact sequence of graded vector spaces
\[
 0\longrightarrow \mathcal M(F)(-1)
 \longrightarrow \mathcal M(G)
 \longrightarrow \bigoplus_{k\ge0}P_k(G)
 \longrightarrow0,
\]
and hence (using the notation $\mathfrak m_G$ for the matching
generating polynomial of $G$)
\[
 \Hilb(\mathcal D(G),t)
 =\mathfrak m_G(t)-t\mathfrak m_F(t);
\]
see \cref{thm:unicyclic-hilbert}.  Minimal pairs with the same matching generating
polynomial but different Dunkl Hilbert series occur on $4$ vertices in
general, on $5$ vertices among connected graphs, and on $6$ vertices among
$2$-connected graphs; see
\cref{prop:matching-polynomial-counterexample,rmk:matching-counterexample-minimality}.

\item
Finally, we survey a related construction in the larger square-zero edge
algebra \cite{PS}, where distinct incident edge generators are allowed to multiply.
For each nontrivial edge cut $\delta(S)$, consider the cut monomial
\[
 \eta_S:=\prod_{e\in\delta(S)}u_e,
\]
and impose the relations $\eta_S=0$ for all such $S$.  If these relations
are imposed inside the subalgebra generated by the signed incidence sums,
the resulting quotient is the central graphical zonotopal algebra.  If they
are imposed instead in the whole square-zero edge algebra, the surviving
monomials are precisely the $u_F$ for which deleting $F$ leaves $G$
connected; hence the Hilbert series counts connected spanning subgraphs and
is a specialization of the all-terminal reliability polynomial.  See
\cref{prop:central-quotient,prop:reliability-quotient}; these results are
not new, so we provide no proofs.
\end{itemize}

\medskip
\textbf{Motivation.}
This work is inspired by the
Fomin--Kirillov approach to Schubert calculus, see \cite{FK, Bazlov, KM}.  However in our set-up the adjective ``Dunkl'' is motivational only.  Unlike the
Fomin--Kirillov algebra~\cite{FK}, our matching algebra is commutative and
kills every product of two incident edge generators. 

Recall that the
Fomin--Kirillov quadratic algebra $\mathcal E_n$ is generated by
antisymmetric symbols $[ij]=-[ji]$ subject, in particular, to
$[ij]^2=0$, commutation of generators with disjoint supports, and
certain three-term relations on triples of indices.  Its distinguished
elements
\[
   \vartheta_i=\sum_{j\ne i}[ij]
\]
are known as the Dunkl elements.  A fundamental feature of the construction is
that these elements commute, and the commutative subalgebra which they
generate provides an algebraic model for the cohomology ring of the
complete flag variety; see~\cite{FK}.  This point of view has subsequently
led to variants for other Coxeter groups, quantum cohomology, $K$-theory,
and Nichols--Woronowicz algebras.

Blasiak, Liu, and M\'esz\'aros~\cite[Definition~2.3]{BLM} have
introduced restrictions of this algebra to a graph.  For a graph $G$
on $[n]$, they define $\mathcal E_G$ to be the subalgebra of
$\mathcal E_n$ generated by those $[ij]$ for which $\{i,j\}$ is an edge
of $G$.  Their graph restriction is thus a subalgebra construction,
rather than the quotient obtained by setting nonedge generators equal
to zero.

Our construction may be viewed as a
quotient of this picture.  Give the generator
corresponding to an unoriented edge $\{i,j\}$ the sign determined by an
orientation and write it as $u_{ij}=-u_{ji}$.  If, in addition to the
Fomin--Kirillov relations, one kills every product
\[
     u_{ij}u_{k\ell}
     \qquad\text{with}\qquad
     \{i,j\}\cap\{k,\ell\}\ne\varnothing,
\]
and kills the generators corresponding to nonedges of $G$, then the
resulting commutative quotient is precisely the matching algebra
$\mathcal M(G)$.  Under this quotient the Dunkl element
$\vartheta_i$ descends to the signed incidence sum $\theta_i$.
Thus $\mathcal D(G)$ can be interpreted as the part of the
Dunkl-element construction which survives after only products supported
on matchings are retained.
Equivalently, the quotient map
$\mathcal E_n\twoheadrightarrow\mathcal M(K_n)$ that kills products of
incident edge generators restricts to a surjection
$\mathcal E_G\twoheadrightarrow\mathcal M(G)$, where $\mathcal M(G)$ is
identified with the subalgebra of $\mathcal M(K_n)$ generated by the
edges of $G$.

This truncation changes the nature of the problem considerably.
The flag-variety and Bruhat-order connections of the classical setting
are replaced by the incidence structure of a graph, while the linear
forms $x_i-x_j$, i.e., the roots of type $A$, reappear on the dual side.
One may therefore ask how much of the structure carried by Dunkl
elements survives under this severe matching truncation.  The results
below show that the answer retains a surprising amount of structure:
for the complete graph the homogeneous components are governed by
two-row Specht modules, whereas for arbitrary graphs the failure or
success of spanning the corresponding Specht module detects
connectivity properties of the graph.  In this sense, the matching
quotient provides a bridge between the Dunkl-element philosophy,
matching theory, and the representation theory of the symmetric group.

\medskip
\textbf{Related work.} 
Subalgebras generated by vertices in the commutative algebra
 generated by edges have been considered in e.g., \cite{PS, PShSh}.
Moreover, in the case of complete graphs they also have a geometric
interpretation, see \cite{ShSh3}.  
Products of arbitrary subfamilies of a vector configuration have a rich
theory related to Tutte polynomials; see Berget~\cite{Berget} and the
zonotopal literature~\cite{AP,HR}.  Our restriction to vertex-disjoint
subfamilies is not matroidal: it remembers the incidence of edges at vertices.
It also differs from graph Specht modules and their relationship with matching
polytopes studied by Liu~\cite{LiuMatching}.  The construction lies naturally
between these two subjects. 

Several nearby constructions attach graded objects to matchings but use a
different algebraic mechanism.  Numata~\cite{Numata} studies an Artinian
Gorenstein algebra obtained by apolarity from a weighted matching generating
function and proves a strong Lefschetz property.  Li~\cite{Li} proves strong
equivariant log-concavity for the graded permutation representation on the
matchings of an arbitrary graph.  For perfect matchings of complete graphs,
Filmus and Lindzey~\cite{FL} develop harmonic polynomial models involving
Specht polynomials and matching-inclusion matrices.
Bergeron, Chan, Soltani, and Zabrocki~\cite{BCSZ} study a
construction deceptively parallel to our $P_k(K_n)$ in an exterior algebra: products of differences
$\xi_j-\xi_i$ indexed by noncrossing pairings give bases of exterior
quasisymmetric harmonic spaces whose degree-$k$ dimensions agree with $\dim P_k(K_n)$.  We compare the two constructions more closely in
\cref{rem:exterior-analogue}.  None of these constructions is the
incidence-generated subalgebra $\mathcal D(G)\subseteq\mathcal M(G)$ studied
here.

There is one direct overlap at the level of the abstract complete-graph
presentation.  Jonsson Kling, Lundqvist, Mohammadi, Orth, and
S\'{a}enz-de-Cabez\'{o}n~\cite{JKLMOS} study the more general almost complete
intersections
\[
 \kk[z_1,\ldots,z_n]/(z_1^2,\ldots,z_n^2,(z_1+\cdots+z_n)^q).
\]
The quotient occurring in our characteristic-zero presentation of
$\mathcal D(K_n)$ is their case $q=1$.
The contribution here is therefore not the abstract quotient itself, but its
realization inside a matching algebra, its matching-product duality, and the
resulting saturation problem for arbitrary graphs.

\medskip
\textbf{AI use statement.}
The bulk of this paper was written by GPT-5.6 based on old notes
(which contained more questions than proofs) and some tactical prompting;
in particular, GPT-5.6 contributed all the proofs in Section 5.
We have read and verified the entire text and edited it for
clarity and style.

\medskip
\textbf{Acknowledgments.}
We thank Per Alexandersson for inspiring conversations early on in the
project.

\medskip
\textbf{Notations.}
Throughout, $\kk$ is a field.
In Remark~\ref{rem:exterior-analogue}, Theorem~\ref{thm:presentation}, Section~\ref{sec:log-concavity} and Appendix~\ref{sec:cut-quotients}, we will additionally assume that $\kk$ has characteristic zero.

\section{The matching algebra}

Let $G = (V, E)$ be a finite simple graph.%
\footnote{For background on elementary graph theory, including paths,
matchings, and connectivity, we refer to \cite{Lavrov}.}
Its edges are $2$-element sets of vertices;
thus, two edges $e,f$ satisfy $e\cap f \ne\varnothing$
if and only if they have a common endpoint (this allows $e=f$).
Two edges $e$ and $f$ are said to be \emph{incident} if
$e\cap f \ne\varnothing$.

\begin{definition}
The \emph{matching algebra} of $G$ is the commutative $\kk$-algebra
\[
 \mathcal M(G):=
 \kk[u_e:e\in E]/\bigl(u_eu_f:e\cap f\ne\varnothing\bigr).
\]
That is, $\mathcal M(G)$ is the quotient of the polynomial ring
$\kk[u_e:e\in E]$ by the monomial ideal generated by all products
$u_eu_f$ with $e$ and $f$ being two incident edges
(including the case $e=f$).
It is graded by degree, where every generator $u_e$ has degree $1$.
\end{definition}

For a matching $M\subseteq E$ of $G$, define the element
\[
 u_M:=\prod_{e\in M}u_e \in \mathcal M(G),
 \qquad \text{ in particular }
 u_\varnothing:=1.
\]

\begin{proposition}\label[proposition]{prop:matching-basis}
The elements $u_M$, indexed by the matchings $M$ of $G$, form a homogeneous
$\kk$-basis of $\mathcal M(G)$.  Consequently, the Hilbert series of
$\mathcal M(G)$ is
\[
 \Hilb(\mathcal M(G),t)=\sum_{k\geq0}m_k(G)t^k,
\]
where $m_k(G)$ is the number of $k$-matchings of $G$.
\end{proposition}

\begin{proof}
The defining ideal $\bigl(u_eu_f:e\cap f\ne\varnothing\bigr)$ is monomial.
Its standard monomials are the square-free
products $u_{e_1} u_{e_2} \cdots u_{e_k}$ of generators
such that no two of the edges $e_1, e_2, \ldots, e_k$ have
a common endpoint; these are exactly the $u_M$
corresponding to the matchings $M$ of $G$.
Thus, these $u_M$ (or, rather, their residue classes) form
a basis of the quotient algebra $\mathcal M(G)$.
The formula for the Hilbert series follows directly from this.
\end{proof}

Thus, for the complete graph $K_n$, we have
\[
 \Hilb(\mathcal M(K_n),t)
 =\sum_{0\leq k\leq n/2}
 \frac{n!}{2^k k!(n-2k)!}t^k,
\]
and the complete bipartite graph $K_{a,b}$ satisfies
\[
 \Hilb(\mathcal M(K_{a,b}),t)
 =\sum_{k\geq0}k!\binom ak\binom bk t^k.
\]
In particular, $\dim\mathcal M(K_n)$ is the number of involutions in $S_n$.

Equivalently, if $\Delta(G)$ denotes the matching complex of $G$ (the
independence complex of the line graph $L(G)$), then $\mathcal M(G)$ is the
Artinian monomial quotient obtained from its Stanley--Reisner ring
$\kk[\Delta(G)]$ by adjoining the squares $u_e^2$ to the Stanley--Reisner
ideal.  Thus $\mathcal M(G)$ itself is not a Stanley--Reisner ring in the
usual sense.  Matching complexes and their face rings are classical; see, for
example, Bj\"orner--Lov\'asz--Vre\'cica--\v{Z}ivaljevi\'c~\cite{BLVZ}
and Jonsson~\cite[Chapter~5]{Jonsson}.  Ferrarello and Fr\"oberg~\cite{FF}
explicitly relate the Hilbert series of the ordinary Stanley--Reisner ring of
a matching complex to graph-subgraph polynomials.

More directly, Dao and Nair~\cite{DaoNair} associate to an arbitrary
simplicial complex $\Delta$ the square-zero face algebra
\[
 A(\Delta):=\kk[x_1,\ldots,x_m]/(I_\Delta,x_1^2,\ldots,x_m^2),
\]
and study its Lefschetz properties.  Thus our matching algebra is exactly
$\mathcal M(G)=A(\Delta(G))$.  Such algebras are also a special case of the
quadratic monomial quotients studied by Migliore, Nagel, and
Schenck~\cite{MNS}.  What we shall study (starting in the next section) is
a certain \textbf{subalgebra} $\mathcal D(G)$ of $\mathcal M(G)$,
which -- to our knowledge -- is new.

\begin{example}\label[example]{ex:matching-p4}
Let $G=P_4$ be the path $1-2-3-4$.  Then
\[
 \mathcal M(P_4)
 =\kk[u_{12},u_{23},u_{34}]/
 (u_{12}^2,u_{23}^2,u_{34}^2,u_{12}u_{23},u_{23}u_{34}).
\]
Thus
\[
 1,\quad u_{12},\quad u_{23},\quad u_{34},\quad u_{12}u_{34}
\]
is a basis of $\mathcal M(P_4)$.
The only nonzero product of two edge generators is
$u_{12}u_{34}$, so
\[
 \Hilb(\mathcal M(P_4),t)=1+3t+t^2,
 \qquad
 (u_{12}+u_{34})^2=2u_{12}u_{34}.
\]
By contrast, for the triangle $K_3$ every two edges meet, so
$\mathcal M(K_3)$ has basis $1,u_{12},u_{13},u_{23}$ and its positive-degree
ideal has square zero.
\end{example}

\section{Signed incidence sums and the Dunkl matching algebra}

Fix an orientation of $G$.  Define \emph{incidence coefficients}
$\varepsilon_{t,e}$ for all $e \in E$ and $t \in V$ as follows:
If an edge $e$ is oriented from $v$ to $w$, let
\[
\varepsilon_{v,e}=1, \qquad
\varepsilon_{w,e}=-1,
\]
and let all other incidence
coefficients $\varepsilon_{t,e}$ (with $t \in V \setminus \{v,w\}$) be zero.

\begin{definition}
For $v\in V$, define the \emph{signed incidence sum}
\[
 \theta_v=\sum_{e\in E}\varepsilon_{v,e}u_e \in \mathcal M(G).
\]
The \emph{Dunkl matching algebra} is the graded subalgebra
$\mathcal D(G)$ of $\mathcal M(G)$ generated by these
signed incidence sums. In other words:
\[
 \mathcal D(G):=\kk[\theta_v:v\in V]\subseteq\mathcal M(G).
\]
\end{definition}

Reversing the orientation of an edge amounts to replacing the corresponding
generator $u_e$ by $-u_e$; hence the isomorphism type of $\mathcal D(G)$ is
orientation-independent.  We have
\begin{equation}\label{eq:elementary-relations}
 \theta_v^2=0,
 \qquad
 \sum_{v\in C}\theta_v=0
\end{equation}
for every vertex $v$ and every connected component $C$ of $G$.

Consider the polynomial ring $\kk[x_v \mid v \in V]$.  If an edge $e$ is
oriented as $e=(i,j)$, define the linear polynomial $p_e:=x_i-x_j$.
For a matching $M$ of $G$, define the polynomial
\begin{equation}\label{eq:def-pM}
 p_M:=\prod_{e\in M}p_e
 =\prod_{e=(i,j)\in M}(x_i-x_j),
 \qquad \text{ so in particular } p_\varnothing:=1.
\end{equation}
Reversing the orientation of an edge changes the corresponding $p_e$ and
$p_M$ only by a sign, so all spans considered below are
orientation-independent.

For each integer $k \geq 0$, we let $\Match_k(G)$ denote the set of
all $k$-matchings of $G$ (that is, of all matchings with
exactly $k$ edges), and we define the $\kk$-vector subspace
\begin{equation}\label{eq:def-Pk}
 P_k(G):=\Span\{p_M:M\in\Match_k(G)\}
 \qquad \text{of $\kk[x_v \mid v \in V]$.}
\end{equation}

Consider the $\kk$-vector space $\kk^V$ with basis $V$; its standard
basis vectors will be denoted by $e_v$.
Let $S^k(\kk^V)$ denote its $k$th symmetric power.  Multiplication gives a
linear map
\[
 \mu_k:S^k(\kk^V)\longrightarrow \mathcal M(G)_k,
 \qquad
 e_{v_1}\cdots e_{v_k}\longmapsto
 \theta_{v_1}\cdots\theta_{v_k},
\]
whose image is precisely the $k$-th graded component $\mathcal D(G)_k$
of $\mathcal D(G)$.

\begin{theorem}[Dunkl-polynomial duality]\label{thm:product-realization}
For every $k\geq0$, we have
\begin{align}\label{eq:dimD=dimP}
 \dim\mathcal D(G)_k=\dim P_k(G).
\end{align}
Moreover, for $\alpha=(\alpha_v)_{v\in V}\in\mathbb{N}^V$ with
$|\alpha|=k$, write
\[
 e^\alpha:=\prod_{v\in V}e_v^{\alpha_v} \in S^k(\kk^V),
 \qquad
 x^\alpha:=\prod_{v\in V}x_v^{\alpha_v} \in \kk[x_v \mid v \in V].
\]
The elements $e^\alpha$ form a basis of $S^k(\kk^V)$.  Let
$(e^\alpha)^*$ denote the dual basis, and define the linear isomorphism
\[
 \iota_k:S^k(\kk^V)^*\longrightarrow\kk[x_v:v\in V]_k,
 \qquad
 (e^\alpha)^*\longmapsto x^\alpha.
\]
Then, for every $M\in\Match_k(G)$, we have
\begin{align}\label{eq:iotamuu}
 \iota_k\bigl(\mu_k^*(u_M^*)\bigr)=p_M,
\end{align}
where $u_M^* \in (\mathcal M(G)_k)^*$ denotes extraction of
the $u_M$-coefficient (i.e., the $\kk$-linear map that sends each
polynomial in $\mathcal M(G)_k$ to its $u_M$-coefficient).
\end{theorem}

\begin{proof}
Fix a $k$-matching $M$.  We begin by proving \eqref{eq:iotamuu}.
Observe that the map $\iota_k$ sends each form $f \in S^k(\kk^V)^*$
to the polynomial $\sum_{|\alpha| = k} f(e^\alpha) x^\alpha$.
Hence, in order to prove \eqref{eq:iotamuu}, it suffices to show
that
\begin{equation}\label{eq:iotamuu-proof}
 \bigl(\mu_k^*(u_M^*)\bigr)(e^\alpha)=[x^\alpha]p_M
 \qquad\text{for every }\alpha \in \mathbb{N}^V
 \text{ satisfying }|\alpha| = k.
\end{equation}

So let us fix a family
$\alpha=(\alpha_v)_{v\in V}\in\mathbb{N}^V$ with
$|\alpha|=k$. Then,
\[
 \bigl(\mu_k^*(u_M^*)\bigr)(e^\alpha)
 = u_M^* \left( \mu_k(e^\alpha)\right)
 = u_M^* \left( \prod_{v\in V}\theta_v^{\alpha_v} \right)
 = [u_M]\prod_{v\in V}\theta_v^{\alpha_v}
\]
(where $[u_M] r$ denotes the coefficient of $u_M$ in some
$r \in \mathcal M(G)$).
In order to prove \eqref{eq:iotamuu-proof}, we must show
that this scalar is exactly the coefficient $[x^\alpha]p_M$.

First suppose that $\alpha_v\geq2$ for some vertex $v$.  Then
$\prod_{v\in V}\theta_v^{\alpha_v}=0$, since $\theta_v^2=0$ by
\eqref{eq:elementary-relations}.  Hence the left-hand side of
\eqref{eq:iotamuu-proof} is zero.  On the other hand, $p_M$ is squarefree
in the variables $x_v$, since no vertex belongs to two edges of $M$.
Thus $[x^\alpha]p_M=0$ as well.
So the equality \eqref{eq:iotamuu-proof} holds in this case because
both of its sides vanish.

It remains to consider the case when $\alpha_v\leq1$ for each vertex $v$.
Let
\[
 S:=\{v\in V:\alpha_v=1\},
\]
so that $|S|=|\alpha|=k$ and
\begin{align}
 \label{e-v-prod-1}
 \prod_{v\in V}\theta_v^{\alpha_v}
 =\prod_{v\in S}\theta_v
 =\prod_{v\in S}\left(\sum_{e\in E}\varepsilon_{v,e}u_e\right).
\end{align}
The coefficient of $u_M$ in this product is zero unless every edge of $M$
has an endpoint in $S$.  This condition is equivalent to saying that
every edge of $M$ has \textbf{exactly one} endpoint in $S$
(because the $k$ edges of $M$ are pairwise disjoint,
and thus choosing one endpoint of each of them already gives $k=|S|$
distinct vertices, leaving no space for second endpoints).

Assume this condition holds.  Then there is a unique way to obtain $u_M$
when expanding the product \eqref{e-v-prod-1}:
for each $e\in M$, choose $u_e$ from the factor
indexed by the unique vertex $v\in e\cap S$.  The coefficient of $u_M$
in \eqref{e-v-prod-1} is therefore
\begin{align}
 \label{e-v-prod-2}
 \prod_{\substack{e\in M;\\ v\text{ is the unique element of }e\cap S}}
 \varepsilon_{v,e}.
\end{align}
But the coefficient of $x_v$ in $p_e$ is exactly $\varepsilon_{v,e}$.
Consequently the same product \eqref{e-v-prod-2} is the coefficient of
$x^\alpha=\prod_{v\in S}x_v$ in
\[
 p_M=\prod_{e\in M}p_e.
\]
Thus \eqref{eq:iotamuu-proof} holds in the $\alpha_v \leq 1$ case,
at least under the condition that every edge
of $M$ has an endpoint in $S$. But \eqref{eq:iotamuu-proof}
also holds if this condition is violated, since both sides vanish
in this case. Hence \eqref{eq:iotamuu-proof} always holds.

Forget that we fixed $\alpha$. We thus have proved
\eqref{eq:iotamuu-proof}, which in turn completes
the proof of \eqref{eq:iotamuu}.

Now forget that we fixed $M$.
The forms $u_M^*$ form a basis of $\mathcal M(G)_k^*$,
whereas the $p_M$ span $P_k(G)$. Hence, the equality
\eqref{eq:iotamuu} entails
\[
 \iota_k(\im\mu_k^*) = P_k(G),
\]
so that $P_k(G) = \iota_k(\im\mu_k^*) \cong \im \mu_k^*$
(since $\iota_k$ is an isomorphism).
Therefore
\[
 \dim P_k(G)
 = \dim \im \mu_k^*
 =\rank\mu_k^*
 =\rank\mu_k
 =\dim\im\mu_k
 =\dim\mathcal D(G)_k,
\]
as claimed.
\end{proof}

\begin{remark}\label[remark]{rmk:duality}
The proof gives a perfect pairing
\[
 \mathcal D(G)_k\times P_k(G)\longrightarrow\kk.
\]
Indeed, for any linear map $T:X\to Y$ between two finite-dimensional
$\kk$-vector spaces $X$ and $Y$, the natural pairing
\[
\im T \times \im T^*
\to \kk,
\qquad
(Tx, T^*\varphi) \mapsto \varphi(Tx) = (T^*\varphi)(x)
\]
between
$\im T$ and $\im T^*$ is perfect, since
\[
 \im T\cong X/\ker T,
 \qquad
 \im T^*=(\ker T)^\perp\cong (X/\ker T)^*.
\]
In the present case (where $X = S^k(\kk^V)$ and $Y = \mathcal M(G)_k$
and $T = \mu_k$), we thus obtain a pairing
$\langle \cdot, \cdot \rangle : \mathcal D(G)_k \times P_k(G) \to \kk$
given explicitly by
\begin{equation}\label{eq:coefficient-pairing}
 \langle d,p_M\rangle=[u_M]d
 \qquad \text{for all $d\in\mathcal D(G)_k$ and $M\in\Match_k(G)$}
\end{equation}
(by \eqref{eq:iotamuu}).
Thus, the spaces $\mathcal D(G)_k$ and $P_k(G)$ are in perfect duality, and in particular
$\dim\mathcal D(G)_k=\dim P_k(G)$, as claimed.
\end{remark}

The equality \eqref{eq:dimD=dimP} 
says that $\Hilb(\mathcal D(G),t)=\sum_k\dim P_k(G)t^k$.
Although $m_k(G)$ counts the chosen generators of $P_k(G)$, their linear
relations contain additional information about $G$.

\begin{example}[Small Dunkl matching algebras]\label[example]{ex:small-dunkl}
Orient every displayed edge from the smaller endpoint to the larger one
(i.e., orient it as $(i,j)$ where $i<j$).

For the path graph $P_4$, we have
\[
 \theta_1=u_{12},\qquad
 \theta_2=-u_{12}+u_{23},\qquad
 \theta_3=-u_{23}+u_{34},\qquad
 \theta_4=-u_{34}.
\]
Hence
\[
 u_{12}=\theta_1,\qquad
 u_{23}=\theta_1+\theta_2,\qquad
 u_{34}=-\theta_4,
\]
so every edge generator $u_e$ already lies in $\mathcal D(P_4)$.  Therefore
\[
 \mathcal D(P_4)=\mathcal M(P_4),
 \qquad
 \Hilb(\mathcal D(P_4),t)=1+3t+t^2.
\]
(The same argument shows $\mathcal D(G)=\mathcal M(G)$ for every forest $G$,
since the incidence matrix of a forest has rank $|E|$.)

For the triangle $K_3$, we have
\[
 \theta_1=u_{12}+u_{13},\qquad
 \theta_2=-u_{12}+u_{23},\qquad
 \theta_3=-u_{13}-u_{23}.
\]
Thus $\theta_1+\theta_2+\theta_3=0$, and all products of two $\theta_i$ vanish
because every pair of edges of $K_3$ is incident.  Consequently
\[
 \mathcal D(K_3)=\kk\oplus\Span\{\theta_1,\theta_2\},
 \qquad
 \Hilb(\mathcal D(K_3),t)=1+2t,
\]
so here $\mathcal D(K_3)$ is a proper subalgebra of $\mathcal M(K_3)$.

The first nontrivial degree-$2$ example is $K_4$.  Put
\[
 m_1=u_{12}u_{34},\qquad
 m_2=u_{13}u_{24},\qquad
 m_3=u_{14}u_{23}.
\]
These are the three degree-$2$ basis elements of $\mathcal M(K_4)$.  Direct
multiplication gives
\[
 \theta_1\theta_2=m_2+m_3,\qquad
 \theta_1\theta_3=m_1-m_3,\qquad
 \theta_2\theta_3=-m_1-m_2.
\]
The three displayed vectors have one linear relation and span a $2$-dimensional
space; since $\theta_4=-(\theta_1+\theta_2+\theta_3)$ and $\theta_i^2=0$, they
span all of $\mathcal D(K_4)_2$.  Hence
\[
 \Hilb(\mathcal D(K_4),t)=1+3t+2t^2.
\]
This is the smallest example in which the matching algebra has more degree-$2$
basis elements than the Dunkl matching algebra has degree-$2$ dimensions.
\end{example}

\begin{proposition}\label[proposition]{prop:basic-properties}
The following properties hold.
\begin{enumerate}
\item If $G=G_1\sqcup G_2$, then
$\mathcal D(G)\cong\mathcal D(G_1)\otimes\mathcal D(G_2)$.
\item We have $\mathcal D(G)_k=0$ for all $k>\nu(G)$, where $\nu(G)$ is the matching
number of $G$ (that is, the size of a largest matching of $G$).
\item We have $\dim\mathcal D(G)_1=|V|-c(G)$, where $c(G)$ is the number of connected components of $G$.
\item If $H$ is a spanning subgraph of $G$, then
\[
P_k(H)\subseteq P_k(G),\qquad
\dim\mathcal D(H)_k\leq\dim\mathcal D(G)_k.
\]
\end{enumerate}
\end{proposition}

\begin{proof}
The first assertion follows from the disjoint sets of variables and incidence
sums.  The second follows from \eqref{eq:dimD=dimP}.  For the third,
observe that $P_1(G)$ is the span of the vectors $x_i-x_j$, which are the
columns of the incidence matrix of $G$; thus, its dimension is the rank of
the incidence matrix, which is known to equal $|V|-c(G)$.
The fourth assertion is immediate (note that there is a graded surjective
$\kk$-algebra homomorphism $\mathcal D(G) \to \mathcal D(H)$ given by sending
all $u_e$ with $e \notin E(H)$ to $0$).
\end{proof}

\begin{remark}[Dependence on the characteristic]
\label[remark]{rmk:characteristic-dependence}
Although \cref{thm:product-realization} holds over every field, the Hilbert
series of $\mathcal D(G)$ can depend on the characteristic.  For example,
let $G=C_5\mathbin{\square}K_2$ be the pentagonal prism, with vertices
$a_0,\ldots,a_4,b_0,\ldots,b_4$ and edges
\[
 a_i a_{i+1},\qquad b_i b_{i+1},\qquad a_i b_i
 \qquad(i\in\mathbb Z/5\mathbb Z).
\]
Then
\begin{equation}\label{eq:prism-characteristic-hilbert}
 \Hilb(\mathcal D(G),t)
 =1+9t+35t^2+75t^3+\alpha_{\kk}t^4+11t^5,
\end{equation}
where $\alpha_{\kk}=79$ if $\operatorname{char}\kk=2$, and
$\alpha_{\kk}=80$ otherwise.
This shows that the inclusion $\mathcal D(G) \to \mathcal M(G)$
need not remain injective after reduction modulo a prime.

\end{remark}

\section{Complete graphs and two-row Specht modules}

We now specialize to the complete graph $K_n$ on the vertex set $[n]$.
The symmetric group $S_n$ acts on $\kk[x_1,\ldots,x_n]$ by permuting the
variables:
\[
 \sigma(x_i)=x_{\sigma(i)}
 \qquad(\sigma\in S_n).
\]
Since $S_n$ also permutes the $k$-matchings of $K_n$, this action preserves
$P_k(K_n)$: more precisely, $\sigma(p_M)=\pm p_{\sigma(M)}$, where the sign
only reflects our choices of orientations for the edges.  Thus $P_k(K_n)$ is
naturally an $S_n$-representation.

To identify this representation, let us recall the polynomial realization of
the Specht modules.  Let $\lambda\vdash n$, and let $T$ be a Young tableau of
shape $\lambda$, filled bijectively with $1,2,\ldots,n$.  For a column of $T$
whose entries, from top to bottom, are $i_1,\ldots,i_r$, form the Vandermonde
factor
\[
 \prod_{1\leq a<b\leq r}(x_{i_a}-x_{i_b}).
\]
The \emph{Specht polynomial} of $T$ is the product of these factors over all
columns of $T$:
\[
 \Delta_T
 :=
 \prod_{\substack{\text{column }C\text{ of }T;\\
                   i_1,\ldots,i_r\text{ are the entries of }C}}
 \ \prod_{1\leq a<b\leq r}(x_{i_a}-x_{i_b}).
\]
The span of the Specht polynomials $\Delta_T$ over all tableaux $T$ of shape
$\lambda$ is the usual Specht module $S^\lambda$.
See, e.g., \cite[Satz 1]{Specht} or \cite[Theorem 5.6.1]{sga}.

Now let $\lambda=(n-k,k)$, where
$0\leq k\leq\lfloor n/2\rfloor$.  Such a diagram has exactly $k$ columns of
height $2$, while all remaining columns have height $1$.  Hence every tableau
$T$ of shape $(n-k,k)$ determines a $k$-matching $M(T)$ of $K_n$, namely the
pairs of entries occurring in its two-box columns, and we have
\[
 \Delta_T=\pm p_{M(T)}.
\]
Conversely, every $k$-matching of $K_n$ arises in this way from some tableau of
shape $(n-k,k)$.  Thus the span $P_k(K_n)$ of the $p_M$'s is precisely the
Specht module of shape $(n-k,k)$.

\begin{theorem}\label{thm:complete-specht}
For every $k$ satisfying $0\leq k\leq\lfloor n/2\rfloor$, we have,
as $S_n$-modules,
\[
 P_k(K_n)\cong S^{(n-k,k)}.
\]
Consequently,
\[
 \Hilb(\mathcal D(K_n),t)
 =\sum_{0\leq k\leq n/2}
 \left(\binom nk-\binom n{k-1}\right)t^k.
\]
\end{theorem}

\begin{proof}
The $S_n$-module identification was established above.
For the Hilbert series, we need to compute $\dim S^{(n-k,k)}$.
Recall that the dimension of any Specht module $S^\lambda$
equals the number $f^\lambda$
of all standard Young tableaux of shape $\lambda$, and can be
computed by the hook-length formula.
Thus,
\[
 \dim S^{(n-k,k)}=
 f^{(n-k,k)} =
 \binom nk-\binom n{k-1}
 =\frac{n-2k+1}{n-k+1}\binom nk.
\]
Together with \eqref{eq:dimD=dimP}, this yields the displayed
Hilbert series.
\end{proof}

\begin{remark}[Positive characteristic]\label[remark]{rmk:complete-poschar}
The identification $P_k(K_n)\cong S^{(n-k,k)}$ and the hook length
formula for the dimension of the Specht module are valid
over any field.  However, if $\kk$ has positive characteristic,
then the Specht module can fail to be irreducible.
For general graphs, in contrast, even the Hilbert series can depend
on the characteristic; see Remark~\ref{rmk:characteristic-dependence}.
\end{remark}

The coefficients of $\Hilb(\mathcal D(K_n),t)$
form the Catalan triangle and count standard Young tableaux
of the fixed shape $(n-k,k)$.

\begin{remark}[An exterior analogue]\label[remark]{rem:exterior-analogue}
Assume in this remark that $\kk$ has characteristic zero.
Bergeron, Chan, Soltani, and Zabrocki~\cite{BCSZ} study a strikingly
parallel construction in the exterior algebra
\[
 R_n=\bigwedge(\xi_1,\ldots,\xi_n).
\]
For a matching\footnote{They refer to matchings as ``pairings''.}
$M=\{(i_1,j_1),\ldots,(i_k,j_k)\}$ of $K_n$, ordered so that
$j_1<\cdots<j_k$ and $i_r < j_r$ for every $r$, they set
\[
 \Delta_M :=
 (\xi_{j_1}-\xi_{i_1})\cdots(\xi_{j_k}-\xi_{i_k}) \in R_n.
\]
They define a specific noncrossing matching $C(\alpha)$
for each ballot sequence $\alpha$.
Their Corollary~4.5 shows that the elements $\Delta_{C(\alpha)}$
indexed by all ballot sequences $\alpha$ (of length $n$ with
$k$ entries equal to $1$) form a basis of the
degree-$k$ exterior quasisymmetric harmonic space $EQH_{n,k}$.
Thus,
\[
 \dim EQH_{n,k}=f^{(n-k,k)}=\dim P_k(K_n).
\]
There is in fact a very explicit vector-space isomorphism between these two
spaces.  The same ballot sequences give noncrossing matchings $C(\alpha)$,
and the commutative polynomials $p_{C(\alpha)}$ are linearly independent:
with respect to the lexicographic order induced by
$x_1>\cdots>x_n$, the lexicographically smallest monomial of
$p_{C(\alpha)}$ is obtained by choosing the right endpoint of every pair
(i.e., the positions of all closing parentheses), and these monomials
are distinct for different $\alpha$.  Since there are
$f^{(n-k,k)}$ such ballot sequences, these polynomials form a basis of
$P_k(K_n)$.  Hence there is a unique isomorphism
\begin{equation}\label{eq:comm-exterior-web-isomorphism}
 \Psi_k:P_k(K_n)\longrightarrow EQH_{n,k},
 \qquad
 p_{C(\alpha)}\longmapsto \Delta_{C(\alpha)}.
\end{equation}
So, at the level of graded vector spaces with their preferred noncrossing
bases, the commutative and exterior pictures are very close.

However, the parallels end here. The vector space $EQH_{n,k}$
is the span of the $\Delta_M$ corresponding to all noncrossing
matchings $M$ (by \cite[Proposition 2.9 and Corollary 4.5]{BCSZ}),
but usually not the span of $\Delta_M$ for \textbf{all} the
matchings $M$.
In particular, it is not an $S_n$-representation, which should
not come as a surprise because the $S_n$-action does not preserve
noncrossingness (unless $k\leq 1$).
Thus, there is no hope for an isomorphism of $S_n$-representations.
\footnote{Another false trail: There is an obvious linear
isomorphism from the squarefree degree-$k$ part of
$\kk[x_1,\ldots,x_n]$ to $\bigwedge\nolimits^k\kk^n$, sending
$x_{a_1}\cdots x_{a_k}$ (with $a_1<\cdots<a_k$) to
$\xi_{a_1}\cdots\xi_{a_k}$.  But this map does not send $p_M$ to
$\Delta_M$: when the chosen endpoints are reordered,
anticommutativity contributes signs which depend on the chosen term of the
expansion.
Thus, for example, $(x_1 - x_4) (x_2 - x_3)
= x_1 x_2 - x_1 x_3 - x_2 x_4 + x_3 x_4$ is sent to
$\xi_1 \xi_2 - \xi_1 \xi_3 - \xi_2 \xi_4 + \xi_3 \xi_4
\neq (\xi_1 - \xi_4) (\xi_2 - \xi_3)$.}

To drive this point home, let us compare the span of the $\Delta_M$
for all $k$-matchings $M$ of $K_n$ with $EQH_{n,k}$.  Let
\[
 W=\Span\{\xi_i-\xi_j:1\leq i,j\leq n\},
\]
the standard $(n-1)$-dimensional representation of $S_n$.
For a $k$-matching $M$, the element $\Delta_M$ belongs to
the subspace $\bigwedge\nolimits^k W\subseteq\bigwedge\nolimits^k \kk^n$
(since the $k$ factors $\xi_{j_r} - \xi_{i_r}$ of $\Delta_M$ all lie in $W$),
and is nonzero (since the $k$ factors of $\Delta_M$
are linearly independent).  The $S_n$-orbit of any one such
$\Delta_M$ consists, up to sign, of all $\Delta_N$s corresponding
to all $k$-matchings $N$.  Since the $S_n$-representation
\[
 \bigwedge\nolimits^k W\cong S^{(n-k,1^k)}
\]
is irreducible in characteristic zero, it follows that
\begin{equation}\label{eq:exterior-matching-span}
 \Span\{\Delta_M:M\in\Match_k(K_n)\}
 =\bigwedge\nolimits^k W\cong S^{(n-k,1^k)}.
\end{equation}
Thus the complete commutative matching family
$\{p_M:M\in\Match_k(K_n)\}$ spans the two-row module
$S^{(n-k,k)}$, whereas the complete exterior matching family
$\{\Delta_M:M\in\Match_k(K_n)\}$ spans the hook
module $S^{(n-k,1^k)}$.  Their dimensions are respectively
\[
 \binom nk-\binom n{k-1}
 \qquad\text{and}\qquad
 \binom{n-1}{k},
\]
which agree only for $k=0,1$.

Already for $n=4$ and $k=2$, the difference between the complete
commutative and exterior matching families is visible in the Garnir relation
\[
 p_{12}p_{34}-p_{13}p_{24}+p_{14}p_{23}=0
\]
in commuting variables.  The three corresponding exterior products are
linearly independent (they span $\bigwedge\nolimits^2W$, which has dimension $3$).
\end{remark}

\begin{theorem}\label{thm:presentation}
Assume that the field $\kk$ has characteristic $0$.
Then, there is an isomorphism of graded algebras
\[
 \mathcal D(K_n)\cong
 \kk[z_1,\ldots,z_n]/
 (z_1+\cdots+z_n,z_1^2,\ldots,z_n^2),
 \qquad z_i\longmapsto\theta_i.
\]
\end{theorem}

As an abstract commutative algebra, the quotient in
\cref{thm:presentation} is the case $q=1$ of the family
\[
 \kk[z_1,\ldots,z_n]/
 (z_1^2,\ldots,z_n^2,(z_1+\cdots+z_n)^q)
\]
studied in detail by Jonsson Kling et al.~\cite{JKLMOS}; among other things,
they determine Gr\"obner bases and Hilbert series for the general family and
give another proof of the strong Lefschetz property of the squarefree
algebra.  We nevertheless include the short Boolean-lattice argument needed
for the present case.

\begin{proof}
The relations in the ideal $(z_1+\cdots+z_n,z_1^2,\ldots,z_n^2)$
hold for the elements $\theta_1,\ldots,\theta_n$ by
\eqref{eq:elementary-relations}, so there is a surjective
$\kk$-algebra homomorphism
\[
 \Phi : \kk[z_1,\ldots,z_n]/
 (z_1+\cdots+z_n,z_1^2,\ldots,z_n^2)
 \to \mathcal D(K_n)
\]
that sends each $z_i$ to $\theta_i$.
We must show that this surjection $\Phi$ is an isomorphism.
It clearly suffices to show that its domain and its codomain
have equal Hilbert functions.  Put
\[
 A:=\kk[z_1,\ldots,z_n]/(z_1^2,\ldots,z_n^2)
 \qquad\text{and}\qquad
 L:=z_1+\cdots+z_n \in A.
\]
Note that the $\kk$-vector space $A$ has dimension $2^n$
and a basis consisting of the (residue classes of the)
squarefree monomials
\[
z_S:=\prod_{i\in S}z_i
\qquad \text{ for all } S \subseteq [n].
\]
We must compute the Hilbert function of $A/(L)$.
This boils down to determining $\dim(A/(L))_k$ for all $k$.
This is rather well-known; indeed, the maximal-rank statement below is
Gottlieb's classical rank theorem for the Boolean up-operator
\cite[Corollary 2]{Gottlieb}
(restricted to consecutive ranks)
underlying the strong Lefschetz property of this squarefree
complete intersection; see also Stanley~\cite{StanleyLefschetz},
\cite[\S3.1]{JKLMOS}, \cite{WilsonIncidence},
\cite[Theorem 4.7]{StanleyAC}, \cite[Corollary 15.8.5]{Godsil},
\cite[Corollary 5.26]{Burcroff},
and many other places.  We give the elementary argument here for the
sake of completeness.  We claim that
multiplication by $L$ -- that is, the $\kk$-linear map
\[
 U_r:A_r\longrightarrow A_{r+1},\qquad f\longmapsto Lf
\]
-- has maximal rank in every degree (i.e., is injective for $r < n/2$
and surjective for $r \geq n/2$).  It is enough to prove this over
$\mathbb R$.  Indeed, in the squarefree monomial bases, the matrix of $U_r$
has only $0$'s and $1$'s, so a maximal minor which is nonzero over
$\mathbb R$ is a nonzero integer and therefore remains nonzero over every
field of characteristic zero.

The maps $U_r : A_r \to A_{r+1}$ for all $r$ are the graded pieces of a
$\kk$-linear map $U : A \to A$, which we call the \emph{up-operator}
since it raises degrees by $1$.
Define the \emph{down-operator} $D : A \to A$ as the $\kk$-linear map
given by
\[
 D(z_S):=\sum_{i\in S}z_{S\setminus\{i\}}
 \qquad \text{ for each } S \subseteq [n].
\]
This operator decreases degrees by $1$; that is, $D (A_r) \subseteq
A_{r-1}$ for each $r$.

Give the $\mathbb{R}$-vector space $A$ the inner product for which the squarefree
monomials $z_S$ are orthonormal.
Then $D$ is the adjoint of $U$.
A direct calculation gives, on $A_r$, the equality
\begin{equation}\label{eq:boolean-commutator}
 DU-UD=(n-2r)\,\operatorname{id}_{A_r}.
\end{equation}
Indeed, for $|S|=r$, the mixed terms
$z_{S\setminus\{i\}\cup\{j\}}$ occur once in both $DU(z_S)$ and $UD(z_S)$,
while $z_S$ occurs respectively $n-r$ and $r$ times.

If $r<n/2$ and $0\ne f\in A_r$, then
\begin{align*}
\|Uf\|^2
&= \langle Uf, Uf\rangle
= \langle DUf,f\rangle \qquad \left(\text{since $D$ and $U$ are adjoint}\right) \\
&= \langle UDf + (n-2r)f,f\rangle
 \qquad \left(\text{since } DU=UD+(n-2r)\,\operatorname{id}_{A_r} \text{ by } \eqref{eq:boolean-commutator}\right) \\
&= \langle UDf,f\rangle+(n-2r)\|f\|^2 \\
&=\|Df\|^2+(n-2r)\|f\|^2
 \qquad\left(\text{again since $D$ and $U$ are adjoint}\right)
\\
&>0 \qquad \left(\text{by positive definiteness and since $n-2r>0$}\right),
\end{align*}
whence $Uf \neq 0$.
Thus $U_r$ is injective on the lower half of $A$ (that is, for $r < n/2$).
On the other hand, the adjoint of $U_r$ with respect to our inner
product is $D_{r+1}:A_{r+1}\to A_r$ (that is, the restriction of $D$
to $A_{r+1}$); thus, the injectivity of $U_r$ for $r < n/2$
entails the surjectivity of its adjoint $D_{r+1}$ for $r < n/2$.
Under the complement identification
$z_S\leftrightarrow z_{[n]\setminus S}$, this down-operator $D_{r+1}$ becomes
$U_{n-r-1}$ (to make this more precise: we have
$U_{n-r-1} = X_{r} \circ D_{r+1} \circ X_{r+1}^{-1}$, where $X_k$
denotes the isomorphism $A_k \to A_{n-k}$ sending each
$z_S$ to $z_{[n] \setminus S}$).
Hence the surjectivity of $D_{r+1}$ we just proved entails that
$U_{n-r-1}$ is surjective for $r < n/2$.
Substituting $n-r-1$ for $r$ here, we conclude that $U_r$ is
surjective for all $r > n/2-1$, hence in particular for all $r \geq n/2$.

Thus we have shown that $U_r$ is injective whenever $r < n/2$ and
surjective whenever $r \geq n/2$. Hence, $U_r$ always has maximal
rank:
\[
\rank U_r = \min\left\{ \dim A_r, \dim A_{r+1} \right\}
= \min\left\{ \binom nr, \binom n{r+1} \right\}
\]
(since $\dim A_k=\binom nk$). This proves the claim.

For $k=0$, we plainly have $\dim(A/(L))_0=1$.  Now let $k\geq1$.
Since $\dim A_k=\binom nk$, the claim we just proved shows that
\[
 \rank U_{k-1}
 =\min\left\{\binom n{k-1},\binom nk\right\}.
\]
But the degree-$k$ component of $A/(L)$ is the
cokernel of $U_{k-1}:A_{k-1}\to A_k$.  Therefore
\begin{align*}
 \dim(A/(L))_k
 &= \binom nk - \rank U_{k-1}
 = \binom nk - \min\left\{\binom n{k-1}, \binom nk\right\}
 \\
 &=\max\left\{\binom nk-\binom n{k-1},0\right\}.
\end{align*}
This agrees with $\dim\mathcal D(K_n)_k$ by
\cref{thm:complete-specht}.  The surjection $\Phi$
is therefore an isomorphism in every degree (since a
surjective linear map between finite-dimensional
vector spaces of the same dimension must be an
isomorphism).
\end{proof}

\begin{remark}[The presentation in positive characteristic]
\label[remark]{rmk:presentation-poschar}
The characteristic-zero hypothesis in \cref{thm:presentation} is crucial.
Over an arbitrary field, the same assignment $z_i\mapsto\theta_i$ gives a
graded surjection
\[
 \Phi:\kk[z_1,\ldots,z_n]/
 (z_1+\cdots+z_n,z_1^2,\ldots,z_n^2)
 \longrightarrow\mathcal D(K_n).
\]
Suppose now that $\operatorname{char}\kk=p>0$.  In the squarefree monomial
bases, the map
\[
 U_r:A_r\longrightarrow A_{r+1}
\]
is the inclusion matrix between the $r$-subsets and the $(r+1)$-subsets of
$[n]$.  Wilson's diagonal form for inclusion matrices~\cite{WilsonIncidence}
shows, when $r<n/2$, that its nonzero diagonal entries over $\mathbb{Z}$ are
\[
 1,2,\ldots,r+1
 \qquad \left(\text{with positive multiplicities}\right).
\]
Hence $U_r$ is injective over $\kk$ if and
only if $p>r+1$.  By complementing subsets, the corresponding statement in
the upper half is equivalent to surjectivity.  Consequently all the maps
$U_r$ have maximal rank if and only if
\[
 p>\left\lceil\frac n2\right\rceil.
\]
Since \cref{thm:complete-specht} shows that the Hilbert
series of $\mathcal D(K_n)$ itself is independent of the characteristic,
the surjection $\Phi$ above is an isomorphism exactly when
\[
 p>\left\lceil\frac n2\right\rceil;
\]
for $p\leq\lceil n/2\rceil$, its source has a larger Hilbert function
than $\mathcal D(K_n)$ in at least one degree.
\end{remark}

\section{Arbitrary graphs}

We begin with conventions.
For a finite set $W$, let $K_W$ denote the complete graph with vertex
set $W$.  Thus $K_{[n]}=K_n$.
The graphs $K_n$ and $K_W$ are isomorphic whenever $\left|W\right| = n$;
thus, they differ only in the labeling of the vertices. Thus, for the
rest of this section, we \textbf{assume that the vertex set $V$ of
our graph $G$ is $[n] = \left\{1,2,\ldots,n\right\}$}, but we expect the
reader to understand that all our results hold equally well, after relabeling,
for graphs on any other $n$-element vertex set.

The spaces $P_k(G)$ can be computed in finite time: enumerate the
$k$-matchings, expand their products in the monomial basis of
$\kk[x_v]_k$, and take their span (which is the column space of
a matrix).
In particular, $P_k(G) \subseteq P_k(K_n)$, since $G$ is a
subgraph of $K_n$.
We may wonder when equality holds:

\begin{definition}
Let $0\leq k\leq\lfloor n/2\rfloor$.  We say that a graph $G$ on $n$
vertices is \emph{$k$-saturated} if
\[
 P_k(G)=P_k(K_n).
\]
\end{definition}

\begin{problem}\label[problem]{prob:saturation}
Characterize the $k$-saturated graphs, either for a prescribed $k$ or
simultaneously for all $k\leq\nu(G)$.
\end{problem}

Note that this question can be restated as a question about the
Specht generators $\Delta_T$ of the Specht module $S^{(n-k,k)}$:
namely, we ask whether the subcollection selected by the
$k$-matchings of $G$ spans the whole Specht module.
In contrast, the Specht matroid introduced by
Wiltshire-Gordon, Woo, and Zajaczkowska~\cite{WGWZ} records the linear
dependences among the classical Specht generators of a given
Specht module $S^\lambda$.

\subsection{The elementary Garnir relations}

The elementary Garnir relations already give useful information.  For any
two vertices $i$ and $j$, write $p_{ij}=x_i-x_j$, so that $p_{ji}=-p_{ij}$.

\begin{lemma}[Elementary Garnir relations]\label[lemma]{lem:garnir}
Let $N$ be a matching of $G$, and $a,b,c,d$ be four vertices.  Then
\begin{align}
 p_N(p_{ab}+p_{bc}+p_{ca})&=0,
 \label{eq:garnir-triangle}\\
 p_N(p_{ab}p_{cd}-p_{ac}p_{bd}+p_{ad}p_{bc})&=0.
 \label{eq:garnir-square}
\end{align}
\end{lemma}

\begin{proof}
Both identities follow by expanding $p_{ij}=x_i-x_j$.
\end{proof}

When the vertices $a,b,c,d$ in \cref{lem:garnir} are distinct
and are not incident to any edge of $N$, the equalities
\eqref{eq:garnir-triangle} and \eqref{eq:garnir-square}
are linear relations among matching generators of
$P_{|N|+1}(K_n)$ and $P_{|N|+2}(K_n)$, respectively.
These are precisely the three-vertex and four-vertex straightening
relations for the two-row Specht polynomial model.

\subsection{$k$-saturation implies $k$-connectivity}

The first general obstruction is connectivity.
We say that $G$ is \emph{$k$-connected} if $G$ has more than $k$
vertices and the graph $G-S$ is connected for every set
$S\subseteq V$ with $|S|<k$.
(In particular, the graph $G$ is $2$-connected if it is connected,
has at least three vertices and has no cut-vertex.)

\begin{proposition}[Connectivity obstruction]\label[proposition]{prop:k-connected}
Let $1\leq k\leq\lfloor n/2\rfloor$.  If $G$ is $k$-saturated, then $G$ is
$k$-connected.
\end{proposition}

Before we prove this proposition, we state a simple lemma that
we will use twice:

\begin{lemma}[Degree lowering by coefficient extraction]\label[lemma]{lem:xv-extract}
Let $v \in V$ be a vertex of $G$.
Let $k \geq 1$.
Then:

\begin{enumerate}
\item[(a)] Every element of $P_k(G)$ is at most linear in
$x_v$, so coefficient extraction at $x_v$ (while treating
all other indeterminates as constants) defines a linear map
\[
[x_v]:P_k(G)\to P_{k-1}(G - v).
\]

\item[(b)] If the degree of the vertex $v$ in $G$ satisfies
$\deg v \geq 2k-1$,
then this map $[x_v]$ is surjective, so that
$[x_v] P_k(G) = P_{k-1}(G - v)$.
\end{enumerate}
\end{lemma}

\begin{proof}
(a) We must show that $[x_v] p_M \in P_{k-1}(G - v)$ for
every $k$-matching $M$ of $G$.
This can be done directly: If $M$ leaves $v$ unmatched, then
$[x_v] p_M = 0 \in P_{k-1}(G - v)$.
Otherwise, we can remove the edge containing $v$ from the
matching $M$; the result is a $(k-1)$-matching $N$ of $G - v$
that satisfies $[x_v] p_M = \pm p_N \in P_{k-1}(G - v)$.

(b) Assume that $\deg v \geq 2k - 1$.
Then, given any $(k-1)$-matching $M$ of $G - v$, we can find at
least one neighbor $w$ of $v$ that is unmatched by $M$ (since
$M$ matches only $2(k-1) = 2k-2$ vertices, but $v$ has
$\deg v \geq 2k-1$ neighbors). Then, adding the edge $vw$ (which
exists by assumption) to $M$ produces a $k$-matching $M'$ of $G$
such that $p_{M'} = \pm p_M p_{vw}$ satisfies $[x_v] p_{M'} = \pm p_M$,
so that $p_M$ lies in the image of $[x_v]$. This shows that $[x_v]$
is surjective.
\end{proof}

\begin{proof}[Proof of Proposition~\ref{prop:k-connected}]
We induct on $k$.  For $k=1$, the ``products'' $p_M$ are just
the differences $p_e = x_v - x_w$ for all edges $e = (v, w) \in E$.
Thus, $G$ is $1$-saturated if and only if these differences $p_e$
have the same span as \textbf{all} differences $x_i - x_j$ with
$i, j \in V$;
but this holds precisely when $G$ is connected\footnote{For a
proof, observe that if $G$ is connected, then each difference
$x_i - x_j$ can be written as a telescoping sum of the edge
differences $p_e$ for all edges $e$ of a path from $i$ to $j$
(with signs corresponding to their orientations). Thus, the
span of the $p_e$ is the span of all the $x_i - x_j$.
Conversely, if the span of the $p_e$ is the span of all the
$x_i - x_j$, then $G$ must be connected, since a difference
$x_i - x_j$ for two vertices $i$ and $j$ lying in two different
components of $G$ could never be in the span of the $p_e$.
(Alternatively, the same conclusion can be obtained from the
incidence-rank formula in \cref{prop:basic-properties}(3).)
}
(i.e., $1$-connected).

For $k>1$, fix $v\in V$.  
Lemma~\ref{lem:xv-extract}(a) defines a coefficient extraction map
$[x_v]$ with
\[
 [x_v]P_k(G)\subseteq P_{k-1}(G-v).
\]
Moreover, $[x_v]P_k(K_V)=P_{k-1}(K_{V\setminus\{v\}})$
by Lemma~\ref{lem:xv-extract}(b)
(since $n \geq 2k$ and thus $\deg_{K_V} v = n-1 \geq 2k-1$).
Thus $k$-saturation of $G$ implies $(k-1)$-saturation of every $G-v$.
The induction hypothesis applies since $n-1\geq2(k-1)$, and gives
$(k-1)$-connectivity of every $G-v$.

It follows that $G$ is $k$-connected.  Indeed, a nonempty set
$S\subseteq V$ of size less than $k$ contains a vertex $v$, and
$G-S=(G-v)-(S\setminus\{v\})$ is connected.  The graph $G$ itself is
connected as well: any two vertices lie in some connected $G-v$,
since $n\geq4$.  Finally, $n\geq2k>k$.
\qedhere

\end{proof}

For $k\geq3$, $k$-connectivity is not sufficient.

\begin{example}\label[example]{ex:prism}
Let $G$ be the triangular prism.  It is $3$-connected, but it has only
four perfect matchings.  Hence
\[
 \dim P_3(G)\leq4
 <5=\binom63-\binom62
 =\dim P_3(K_6),
\]
so $G$ is not $3$-saturated.
\end{example}

\subsection{$2$-connectivity implies $2$-saturation}

In degree $2$ the connectivity obstruction is sharp.
To prove this, we need some lemmas:

\begin{lemma}[Two-row branching sequence]\label[lemma]{lem:p2-branching}
Let $V$ be a finite set with $|V|\geq4$, let $v\in V$, and put
$V':=V\setminus\{v\}$.  Every element of $P_2(K_V)$ is at most linear in
$x_v$, so coefficient extraction at $x_v$ defines a map
$[x_v]:P_2(K_V)\to P_1(K_{V'})$.  There is a short exact sequence
\begin{equation}\label{eq:p2-branching}
 0\longrightarrow P_2(K_{V'})
 \longrightarrow P_2(K_V)
 \xrightarrow{[x_v]} P_1(K_{V'})
 \longrightarrow0,
\end{equation}
where the first (nonzero) map is the evident inclusion.
\end{lemma}

\begin{proof}
The first map is injective, and its image lies in the kernel
of the last map.
The last map is surjective, by
Lemma~\ref{lem:xv-extract}(b) (applied to $G = K_V$ and $k = 2$,
since $\deg v = |V|-1 \geq 3 = 2\cdot 2 - 1$).
The dimension formula in \cref{thm:complete-specht} gives
\[
 \dim P_2(K_V)-\dim P_2(K_{V'})=|V|-2
 =\dim P_1(K_{V'}),
\]
so that $\dim P_2(K_V) = \dim P_2(K_{V'}) + \dim P_1(K_{V'})$.
All that remains is to apply the easy linear-algebraic lemma
saying that if
\[
 0\longrightarrow A
 \overset{f}{\longrightarrow} B
 \overset{g}{\longrightarrow} C
 \longrightarrow0
\]
is a complex of finite-dimensional vector spaces such that
$f$ is injective and $g$ is surjective and
$\dim B = \dim A + \dim C$, then this complex is exact.
\end{proof}

\begin{lemma}[A degree-two vertex]\label[lemma]{lem:minimal-2connected}
Every finite $2$-connected graph contains a spanning subgraph that is
edge-minimal subject to being $2$-connected.  Every such edge-minimal
subgraph has a vertex of degree $2$.
\end{lemma}

\begin{proof}
The first assertion follows simply by deleting edges as long as
$2$-connectivity is preserved.

For the second, we argue as follows (cf. \cite{DiracMinimal}):
Let $H$ be an edge-minimal $2$-connected spanning subgraph.
Let $v_0v_1\cdots v_m$ be a longest path in $H$.
Every neighbor of $v_0$ lies on this path
(or else the path could be made even longer by starting at this
neighbor).
The vertex $v_0$ cannot have degree smaller than $2$, since this
would violate $2$-connectedness (as we could disconnect $v_0$
by deleting its only neighbor).
If $v_0$ has degree $2$, then we are done.
So let us assume -- for contradiction --
that $v_0$ has degree larger than $2$.
Then, $v_0$ has three neighbors: $v_1$ and two further neighbors
$v_i,v_j$ with $1 < i < j$
(since every neighbor of $v_0$ lies on the path $v_0v_1\cdots v_m$).
Thus, we have a cycle $v_0 v_1 \cdots v_j v_0$ with a chord
$v_0 v_i$.
(A chord of a cycle means an edge joining two non-consecutive
vertices of the cycle.)

Since $H$ is edge-minimal, removing the edge $v_0 v_i$ breaks
$2$-connectedness; in other words, there exists a vertex $w$
such that $(H - v_0 v_i) - w$ is not connected.
That is, there exist two further vertices $x, y \neq w$ that
lie in different components of $(H - v_0 v_i) - w$.
However, the graph $H - w$ is connected (since $H$ is
$2$-connected) and thus must have a path from $x$ to $y$.  Now
we can easily tweak this path so that it does not use the edge
$v_0 v_i$ (while remaining a walk of $H - w$):
Indeed, if it does use this edge, then we can replace this edge
by either the walk $v_0 v_1 \cdots v_i$
or the walk $v_0 v_j v_{j-1} \cdots v_i$ (since one of these
two walks will certainly not contain the deleted vertex $w$).
Thus, we obtain a walk of $(H - v_0 v_i) - w$ from $x$ to $y$.
So $x$ and $y$ do not actually lie in different components of
$(H - v_0 v_i) - w$. This contradiction shows that the case we
are considering (that $v_0$ has degree larger than $2$) cannot
happen at all, and the proof is complete.
\end{proof}

\begin{lemma}[Suppressing a degree-two vertex]\label[lemma]{lem:suppress-degree-two}
Let $H$ be a $2$-connected graph with $|V(H)|\geq4$, and let
$v\in V(H)$ have degree $2$, with neighbors $a,b$.  Put
$V':=V(H)\setminus\{v\}$ and
\[
 H^+:=(H-v)+ab,
 \qquad
 H^-:=H^+-ab=(H-v)-ab.
\]
Thus $H^-$ is obtained from $H-v$ by deleting $ab$ if that edge was already
present; if $ab\notin E(H-v)$, then simply $H^-=H-v$.  Then $H^+$ is
$2$-connected and $H^-$ is connected.
\end{lemma}

\begin{proof}
First consider $H^+$.  It has at least three vertices.  Let $w\in V'$.
We claim that $H^+-w$ is connected.
Indeed, given two vertices of
$H^+-w$, choose a path between them in $H-w$.  If this path uses $v$, then
$v$ is an internal vertex of the path, so both neighbors $a,b$ of $v$ are
present in $H-w$ (since $a$ and $b$ are the only neighbors of $v$ in $H$),
and the path uses the segment $a-v-b$ (or $b-v-a$).
Replacing this segment by the edge $ab$ gives a path in
$H^+-w$ connecting the same two vertices.  Thus $H^+-w$ is connected.
Hence deletion of any vertex leaves $H^+$
connected, so $H^+$ is $2$-connected.

Now consider $H^-$.  Since $H$ is $2$-connected, $H-v$ is connected.  If
$ab\notin E(H-v)$, there is nothing more to prove.  Suppose instead that
$ab\in E(H-v)$ and that deleting $ab$ disconnects $H-v$.  Then $ab$ is a
bridge of $H-v$.  Let $A$ and $B$ be the two components of
$(H-v)-ab$, with $a\in A$ and $b\in B$.  If $A\neq\{a\}$, then in $H-a$
the nonempty set $A\setminus\{a\}$ has no edge to the rest of the graph:
there is no such edge in $H-v-ab$, and the only edges incident with $v$ are
$av$ and $bv$.  This contradicts the connectedness of $H-a$.  Hence
$A=\{a\}$, and similarly $B=\{b\}$.  Thus $V(H)=\{a,b,v\}$, contradicting
$|V(H)|\geq4$.  Therefore $H^-$ is connected.

For example, if $H=C_4$ and $v$ is any vertex, then its neighbors $a,b$ are
not adjacent in $H$.  Thus $H-v$ is a three-vertex path,
$H^+$ is a triangle, and $H^-=H-v$ is that same path.
\end{proof}

\begin{theorem}[Degree-two saturation]\label{thm:degree-two-saturation}
Let $G$ be a graph on $n\geq4$ vertices.  Then
\[
 P_2(G)=P_2(K_n)
 \quad\Longleftrightarrow\quad
 G\text{ is $2$-connected}.
\]
\end{theorem}

\begin{proof}
The forward implication is \cref{prop:k-connected}.  Conversely, let $G$ be
$2$-con\-nected and choose, by \cref{lem:minimal-2connected}, a spanning
edge-minimal $2$-connected subgraph $H\subseteq G$ and a degree-$2$
vertex $v$ of $H$, with neighbors $a,b$.  Put $V':=V\setminus\{v\}$ and
\[
 H^+=(H-v)+ab,
 \qquad
 H^-=(H-v)-ab.
\]
By \cref{lem:suppress-degree-two}, $H^+$ is $2$-connected and $H^-$ is
connected.  If $|V'|=3$, then $H^+=K_{V'}$; otherwise induction on $n$ gives
\[
 P_2(H^+)=P_2(K_{V'}).
\]

We want to prove that $P_2(G)=P_2(K_n)$.  Since
$P_2(H)\subseteq P_2(G)\subseteq P_2(K_V) = P_2(K_n)$, it is enough to prove
$P_2(H)=P_2(K_V)$.

Recall the elementary
linear-algebra fact that if
\[
 0\longrightarrow A\longrightarrow B\longrightarrow C\longrightarrow0
\]
is a short exact sequence of vector spaces,
and if $W\subseteq B$ is a subspace that contains $A$
and surjects onto $C$, then $W=B$.
Applying this to the short exact sequence \eqref{eq:p2-branching}
and the subspace $P_2(H)$ of $P_2(K_V)$, we see that
it suffices to show that $P_2(H)$ contains
$P_2(K_{V'})$ and surjects onto $P_1(K_{V'})$ under $[x_v]$.

We start with the first claim.  Since $P_2(H^+)=P_2(K_{V'})$, it is enough to
show that every matching polynomial $p_M \in P_2(H^+)$ belongs to $P_2(H)$.
So consider a $2$-matching $M$ of $H^+$.
If $ab \notin M$, then $M$ is already a $2$-matching of $H$,
and thus $p_M \in P_2(H)$ follows.  If
$ab \in M$, then we let $cd$ be the other edge of $M$
(so $p_M = \pm p_{ab}p_{cd}$); then, 
the triangle Garnir relation \eqref{eq:garnir-triangle} gives
\[
 p_{ab}p_{cd}
 =(p_{av}+p_{vb})p_{cd}
 =p_{av}p_{cd}+p_{vb}p_{cd},
\]
and both terms on the right come from $2$-matchings of $H$,
so the sum again belongs to $P_2(H)$.
Hence
\[
 P_2(K_{V'})=P_2(H^+)\subseteq P_2(H).
\]

It remains to prove the surjection claim.
For every edge $cd$ of the connected graph
$H^-$, at least one of $a,b$ is not incident with $cd$.  Pairing $cd$ with
the corresponding edge $av$ or $bv$ gives a $2$-matching of $H$ whose
$x_v$-coefficient is, up to sign, $p_{cd}$.  Thus,
$p_{cd} \in [x_v]P_2(H)$.
Therefore
\[
 [x_v]P_2(H)\supseteq P_1(H^-)=P_1(K_{V'}),
\]
where the last equality follows from the connectivity of $H^-$.
That is, $P_2(H)$ surjects onto $P_1(K_{V'})$ under $[x_v]$.
The preceding
linear-algebra observation now yields $P_2(H)=P_2(K_V)$, and hence
$P_2(G)=P_2(K_n)$.
\end{proof}

\subsection{$k$-linkedness implies $k$-saturation}

There is also a simple sufficient condition in every degree.

\begin{definition}
A graph $G$ is \emph{$k$-linked} if, whenever
$a_1,b_1,\ldots,a_k,b_k$ are distinct vertices, there exist pairwise
vertex-disjoint paths $Q_i$ joining $a_i$ to $b_i$ for
$1\leq i\leq k$.
(Paths are defined with respect to the undirected graph $G$; they
need not conform with the orientation we chose for the edges.)
\end{definition}

\begin{proposition}[Linkedness criterion]\label[proposition]{prop:k-linked}
If $G$ is $k$-linked, then $G$ is $k$-saturated.
\end{proposition}

\begin{proof}
Take an arbitrary $k$-matching
$M=\{\{a_1,b_1\},\ldots,\{a_k,b_k\}\}$ of $K_n$, and choose pairwise
vertex-disjoint paths $Q_i$ in $G$ joining $a_i$ to $b_i$.  Along each path,
\[
 x_{a_i}-x_{b_i}
 =\sum_{uv\in E(Q_i)}\varepsilon_{uv}(x_u-x_v),
 \qquad \varepsilon_{uv}\in\{\pm1\}.
\]
Multiplying these $k$ identities expresses $p_M$ as a linear combination of
products obtained by choosing one edge from each $Q_i$.  Since the paths are
vertex-disjoint, every such choice is a $k$-matching of $G$.  Hence
$p_M\in P_k(G)$.  Thus $P_k(K_n)\subseteq P_k(G)$, and the result follows.
\end{proof}

Thus, for $1\leq k\leq\lfloor n/2\rfloor$, we have the useful chain
\[
 k\text{-linked}
 \quad\Longrightarrow\quad
 k\text{-saturated}
 \quad\Longrightarrow\quad
 k\text{-connected}.
\]
The second implication is an equivalence for $k=1,2$, but
\cref{ex:prism} shows that it is not an equivalence for $k=3$.

\subsection{Matching-complement implies $k$-saturation}

The Garnir relations also settle another useful family of graphs.

\begin{theorem}[Matching-complement saturation]\label{thm:matching-complement}
Let $G$ be a graph on $n\geq3$ vertices whose complement $\overline G$ is a
matching.  Then $G$ is $k$-saturated for every
$0\leq k\leq\lfloor n/2\rfloor$.
\end{theorem}

\begin{proof}
Call the edges of $\overline G$ \emph{forbidden}.  We must show that
$p_M \in P_k(G)$ for each $k$-matching $M$ of $K_n$.
We shall show this by strong induction on
the number of forbidden edges in $M$.
There is nothing to prove if $M$ contains no forbidden
edge.  Suppose that $ab\in M$ is forbidden.

If $2k<n$, choose an unmatched vertex $c$.  Since the forbidden edges form a
matching, both $ac$ and $cb$ are edges of $G$.  The triangle relation
\eqref{eq:garnir-triangle} gives
\[
 p_{ab}=p_{ac}+p_{cb}.
\]
After multiplication by the factors belonging to $M\setminus\{ab\}$, this
expresses $p_M$ as a sum of two matching polynomials $p_N$ and $p_R$
where $N$ and $R$ each have one fewer forbidden edge than $M$.
By induction, this shows that $p_M \in P_k(G)$.

If $2k=n$, then $n\geq4$ and $k\geq2$.  Choose another edge $cd\in M$.
Again the matching condition on $\overline G$ implies that all four cross
edges $ac,ad,bc,bd$ belong to $G$.  The four-vertex Garnir relation
\eqref{eq:garnir-square} gives
\[
 p_{ab}p_{cd}
 =p_{ac}p_{bd}-p_{ad}p_{bc}.
\]
Multiplying by the remaining factors again decreases the number of forbidden
edges (by $1$ or by $2$).  Induction completes the proof.
\end{proof}

\begin{corollary}[Multipartite saturation]\label[corollary]{cor:multipartite}
Let $r\geq2$ and let
\[
 G=K_{2,\ldots,2}
\]
be the complete $r$-partite graph with all $r$ parts of size $2$.  Then
\[
 P_k(G)=P_k(K_{2r})\qquad \text{ for all }0\leq k\leq r,
\]
and therefore
\[
 \Hilb(\mathcal D(G),t)
 =\sum_{k=0}^r
 \left(\binom{2r}{k}-\binom{2r}{k-1}\right)t^k.
\]
\end{corollary}

\begin{proof}
The complement of $K_{2,\ldots,2}$ is the perfect matching consisting of the
$r$ within-part pairs, so \cref{thm:matching-complement} applies.
The Hilbert series then follows from \cref{thm:complete-specht}.
\end{proof}

\begin{remark}
The restriction $r\geq2$ is necessary: for $r=1$ the graph has two vertices
and no edges, so it is not saturated in degree $1$.
\end{remark}

For $r=5$, \cref{cor:multipartite} gives
\[
 \Hilb(\mathcal D(G),t)
 = 1+9t+35t^2+75t^3+90t^4+42t^5.
\]

\subsection{Matching-equivalent graphs}

The matching numbers $m_k(G)$ of $G$ are the coefficients of the Hilbert series of
$\mathcal M(G)$, but they do not determine the Hilbert series of
$\mathcal D(G)$.
Let us see when they do agree:

\begin{proposition}[Forests]\label[proposition]{prop:forest-full}
For every graph $G$, we have the equivalence
\[
 \mathcal D(G)=\mathcal M(G)
 \quad\Longleftrightarrow\quad
 G\text{ is a forest}.
\]
Consequently, if $G$ is a forest, then
\[
 \Hilb(\mathcal D(G),t)=\sum_{k\geq0}m_k(G)t^k.
\]
\end{proposition}

\begin{proof}
The degree-one part $\mathcal D(G)_1$ is the row space of an oriented
incidence matrix of $G$, whereas $\mathcal M(G)_1$ has the edge generators
$u_e$ as a basis.  Hence
\[
 \dim\mathcal D(G)_1=|V|-c(G),
 \qquad
 \dim\mathcal M(G)_1=|E|.
\]
If $G$ is a forest, then $|E|=|V|-c(G)$, so the incidence rows span all of
$\mathcal M(G)_1$.  Thus every generator $u_e$ belongs to $\mathcal D(G)$,
and therefore $\mathcal D(G)=\mathcal M(G)$.  Conversely, equality of the
two algebras implies equality in degree one, hence
$|E|=|V|-c(G)$, which is equivalent to every connected component of $G$
being a tree.
\end{proof}

For brevity, write
\[
 \mathfrak m_G(t):=\sum_{k\geq0}m_k(G)t^k
\]
for the matching generating polynomial of $G$.
The next result shows explicitly what information beyond $\mathfrak m_G(t)$
can enter even when $G$ has only one cycle.
We recall that a graph is said to be \emph{unicyclic} if it is
connected and has exactly one cycle.
Equivalently, a unicyclic graph is a graph obtained by adding
an edge to a tree.
Thus, any unicyclic graph has equally many vertices and edges.

\begin{theorem}[Unicyclic graphs]\label{thm:unicyclic-hilbert}
Let $G$ be a unicyclic graph, let $C$ be its unique cycle, and put
$F:=G-V(C)$.  Choose signs $\epsilon_e\in\{\pm1\}$, for $e\in C$, such that
\[
 \sum_{e\in C}\epsilon_e p_e=0,
\]
and put
\[
 c_C:=\sum_{e\in C}\epsilon_e u_e\in\mathcal M(G).
\]
Also set
\[
 \mathcal P(G):=\bigoplus_{k\geq0}P_k(G),
\]
regarded only as a graded vector space, and let
\[
 \rho_G:\mathcal M(G)\longrightarrow\mathcal P(G),
 \qquad
 u_M\longmapsto p_M
\]
be the graded linear map defined on the matching basis of $\mathcal M(G)$.
Then there is a short exact sequence of graded vector spaces
\begin{equation}\label{eq:unicyclic-exact-sequence}
 0\longrightarrow \mathcal M(F)(-1)
 \xrightarrow{\ \cdot c_C\ }
 \mathcal M(G)
 \xrightarrow{\ \rho_G\ }
 \mathcal P(G)
 \longrightarrow0.
\end{equation}
Here $\mathcal M(F)(-1)_k=\mathcal M(F)_{k-1}$.

Consequently, with the convention $m_{-1}(F)=0$, we have
\[
 \dim P_k(G)=m_k(G)-m_{k-1}(F)
 \qquad \text{ for all }k\geq0,
\]
or equivalently,
\[
 \Hilb(\mathcal D(G),t)=\mathfrak m_G(t)-t\mathfrak m_F(t).
\]
\end{theorem}

\begin{proof}
The first map in \eqref{eq:unicyclic-exact-sequence} is well defined because
every edge of $F$ is disjoint from every edge of $C$.
Now, we shall show that it is injective.

Indeed, for any matching $N$ of $F$, we have
\[
 c_Cu_N = \sum_{e\in C}\epsilon_e u_e u_N
 = \sum_{e\in C}\epsilon_e u_{N\cup\{e\}}.
\]
Thus, for any $z=\sum_N a_Nu_N\in\mathcal M(F)$, we have
\[
 c_C z
 = \sum_N \sum_{e\in C} a_N \epsilon_e u_{N\cup\{e\}},
\]
and all the matchings $N\cup\{e\}$ in this double sum are distinct.
Thus the coefficient of $u_{N\cup\{e\}}$ in $c_Cz$ is
$a_N\epsilon_e$.  Since $\epsilon_e\neq0$, the equality $c_Cz=0$
therefore forces $a_N=0$ for every $N$.
Hence, multiplication by $c_C$ really is injective on
$\mathcal M(F)$.

The map $\rho_G$ is surjective by the definition of $\mathcal P(G)$.
Each matching $N$ of $F$ satisfies
\[
 \rho_G(c_Cu_N)
 =p_N\sum_{e\in C}\epsilon_ep_e
 =0.
\]
Thus it remains only to prove exactness in the middle.

Choose an orientation of the edges and define a polynomial ring
\[
 S:=\kk[y_e:e\in E(G)].
\]
Consider the graded $\kk$-algebra homomorphism
\[
 \varphi:S\longrightarrow\kk[x_v:v\in V(G)],
 \qquad
 y_e \longmapsto p_e
 \qquad \left(\text{that is, }
 y_{ij}\longmapsto x_i-x_j\right).
\]
The edge differences $p_e = x_i - x_j$
of a connected graph $G$ span a vector space of dimension
$|V(G)|-1$.  Since $G$ is unicyclic, $|E(G)|=|V(G)|$, so the space of
linear dependences among the differences $p_e=x_i-x_j$, for
$e=(i,j)\in E(G)$, is one-dimensional.  The cycle relation
$\sum_{e\in C}\epsilon_ep_e=0$ is a nonzero such dependence,
and consequently spans all of them.  That is, the
kernel of the degree-one part of $\varphi$ is spanned by
\[
 \ell:=\sum_{e\in C}\epsilon_e y_e.
\]
It follows that
\begin{equation}\label{eq:unicyclic-kernel-phi}
 \ker\varphi=(\ell).
\end{equation}
Indeed, extend $\ell$ to a basis of $S_1$; the images of the remaining basis
vectors are linearly independent, so $S/(\ell)$ is the polynomial ring on
those images.

For each $k \geq 0$, define a $\kk$-vector subspace $U_k$ of $S$ by
\[
 U_k:=\Span_{\kk}\{y_M:M\in\Match_k(G)\},
 \qquad \text{ where }
 y_M:=\prod_{e\in M}y_e,
\]
and put $U:=\bigoplus_{k\geq0}U_k$.  The basis-preserving identification
$u_M\leftrightarrow y_M$ identifies $\mathcal M(G)$ with $U$ as graded vector
spaces, and under this identification $\rho_G$ is the restriction of
$\varphi$ to $U$, while the element $c_C$ of $\mathcal M(G)$ corresponds
to the element $\ell$ of $U$.  Thus exactness in the middle of
\eqref{eq:unicyclic-exact-sequence} is equivalent to
\begin{equation}\label{eq:unicyclic-intersection}
 U\cap(\ell)=\ell\,U(F),
\end{equation}
where
\[
 U(F):=\Span_{\kk}\{y_N:N\text{ is a matching of }F\}.
\]

The inclusion $\ell U(F)\subseteq U\cap(\ell)$ is immediate.  For the reverse
inclusion, take a homogeneous element
\[
 f=\ell q\in U.
\]
Every monomial occurring in $f$ is a matching monomial (i.e., a monomial
of the form $y_M$ for a matching $M$ of $G$), hence is squarefree.
Fix $e\in C$.  Since the coefficient of $y_e$ in $\ell$ is nonzero and $S$
is a domain,
\[
 \deg_{y_e}(\ell q)=\deg_{y_e}q+1
 \qquad(q\neq0).
\]
As $\deg_{y_e}f\leq1$ (because every monomial occurring in $f$
is squarefree), we thus obtain $\deg_{y_e}q=0$.  Doing this for every
$e\in C$ shows that $q$ involves no cycle variables (i.e.,
variables of the form $y_e$ with $e \in C$).

Write
\[
 q=\sum_\alpha c_\alpha y^\alpha.
\]
Because the $y^\alpha$ contain no cycle variables, the monomials
$y_e y^\alpha$, as $(e,\alpha)$ varies with $e\in C$, are pairwise distinct.
Hence there is no cancellation in
\[
 \ell q=\sum_{e\in C}\sum_\alpha
          \epsilon_ec_\alpha y_e y^\alpha.
\]
Since this polynomial lies in $U$, every term with $c_\alpha\neq0$ must be a
matching monomial.  It follows that $y^\alpha=y_N$ for a matching $N$ of
$G$.  Moreover, $N$ cannot meet any vertex of $C$: if an edge of $N$ met
$v\in V(C)$, then choosing a cycle edge $e$ incident with $v$ would make
$y_ey_N$ a non-matching monomial.  Thus $N$ is a matching of
$F=G-V(C)$.  Therefore $q\in U(F)$, proving
\eqref{eq:unicyclic-intersection} and hence the exact sequence.

Taking degree $k$ in \eqref{eq:unicyclic-exact-sequence} gives
a short exact sequence
\[
 0\longrightarrow\mathcal M(F)_{k-1}
 \longrightarrow\mathcal M(G)_k
 \longrightarrow P_k(G)
 \longrightarrow0.
\]
Therefore
\[
 \dim P_k(G)=m_k(G)-m_{k-1}(F).
\]
Finally, summing over $k$ and using
$\dim\mathcal D(G)_k=\dim P_k(G)$ gives
\[
 \Hilb(\mathcal D(G),t)
 =\mathfrak m_G(t)-t\mathfrak m_F(t).
 \qedhere
\]
\end{proof}

\begin{proposition}[Matching-equivalent counterexamples]
\label[proposition]{prop:matching-polynomial-counterexample}
\leavevmode
\begin{enumerate}
\item[(a)] There are two graphs $G$ and $H$ with $4$ vertices each such that
\[
 m_k(G)=m_k(H)\quad\text{for all }k\geq0,
 \qquad \text{ but } \quad
 \Hilb(\mathcal D(G),t)\neq\Hilb(\mathcal D(H),t).
\]
\item[(b)] There are two connected graphs $G$ and $H$ with $5$ vertices each
such that
\[
 m_k(G)=m_k(H)\quad\text{for all }k\geq0,
 \qquad \text{ but } \quad
 \Hilb(\mathcal D(G),t)\neq\Hilb(\mathcal D(H),t).
\]
Moreover, $G$ and $H$ can be chosen with the same degree sequence.
\item[(c)] There are two $2$-connected graphs $G$ and $H$ with $6$ vertices each
such that
\[
 m_k(G)=m_k(H)\quad\text{for all }k\geq0,
 \qquad \text{ but } \quad
 \Hilb(\mathcal D(G),t)\neq\Hilb(\mathcal D(H),t).
\]
Moreover, $G$ and $H$ can be chosen with the same degree sequence.
\end{enumerate}
\end{proposition}

\begin{proof}
\textbf{(a)} Let
\[
 G_4=K_{1,3},
 \qquad
 H_4=K_3\sqcup K_1.
\]
Both graphs have matching numbers
\[
 (m_0,m_1,m_2)=(1,3,0).
\]
However, $G_4$ is connected while $H_4$ has two connected components, so
\cref{prop:basic-properties} gives
\[
 \Hilb(\mathcal D(G_4),t)=1+3t,
 \qquad
 \Hilb(\mathcal D(H_4),t)=1+2t.
\]
\medskip\noindent\textbf{(b)} Let $G_5$ be the triangle
on $\{1,2,3\}$ with the path $3-4-5$ attached, and let $H_5$ be the cycle
$1-2-3-4-1$ with the pendant edge $1-5$.  Thus
\begin{align*}
 E(G_5)=\{12,23,13,34,45\},\qquad
 E(H_5)=\{12,23,34,14,15\}.
\end{align*}
Both graphs are connected, both have degree sequence $(3,2,2,2,1)$, and
both have matching generating polynomial
\[
 \mathfrak m_{G_5}(t)=\mathfrak m_{H_5}(t)=1+5t+4t^2.
\]
For $G_5$, deleting the vertices of the unique triangle leaves the edge
$45$, whereas for $H_5$, deleting the vertices of the unique $4$-cycle
leaves only the isolated vertex $5$.  Hence
\cref{thm:unicyclic-hilbert} gives
\[
 \Hilb(\mathcal D(G_5),t)
 =(1+5t+4t^2)-t(1+t)=1+4t+3t^2,
\]
while
\[
 \Hilb(\mathcal D(H_5),t)
 =(1+5t+4t^2)-t=1+4t+4t^2.
\]
Thus the two graphs have the same matching data and even the same degree
sequence, but different Dunkl Hilbert series.

\medskip\noindent\textbf{(c)} Given a graph $X$ on $n$ vertices, let $X^\ast$ be its cone, obtained by
adjoining one new vertex adjacent to every vertex of $X$.  A $k$-matching of
$X^\ast$ either avoids the new vertex, or consists of a $(k-1)$-matching of
$X$ together with an edge from the new vertex to one of the
$n-2k+2$ vertices left unmatched.  Hence
\begin{equation}\label{eq:cone-matching-numbers}
 m_k(X^\ast)=m_k(X)+(n-2k+2)m_{k-1}(X).
\end{equation}
In particular, coning preserves equality of matching numbers.

Let $G_5,H_5$ be the two graphs constructed in part (b), and add a universal vertex
$6$ to each.  The resulting graphs $G_5^\ast,H_5^\ast$ are $2$-connected,
have the same degree sequence $(5,4,3,3,3,2)$, and by
\eqref{eq:cone-matching-numbers} have matching numbers
\[
 (m_0,m_1,m_2,m_3)=(1,10,19,4).
\]
By \cref{thm:degree-two-saturation}, both have
$\dim P_1=5$ and $\dim P_2=9$.

The four perfect matchings of $G_5^\ast$ give
\[
 p_{12}p_{34}p_{56},\quad
 p_{12}p_{36}p_{45},\quad
 p_{13}p_{26}p_{45},\quad
 p_{16}p_{23}p_{45}.
\]
The last three satisfy the four-vertex Garnir relation
\[
 p_{12}p_{36}p_{45}-p_{13}p_{26}p_{45}
   +p_{16}p_{23}p_{45}=0,
\]
while the first polynomial is independent of their span (for instance its
$x_1x_3x_6$-coefficient is nonzero and the other three have none).  Any two
of the last three are independent, so $\dim P_3(G_5^\ast)=3$.

The four perfect matchings of $H_5^\ast$ give
\[
 p_{12}p_{34}p_{56},\quad
 p_{14}p_{23}p_{56},\quad
 p_{15}p_{23}p_{46},\quad
 p_{15}p_{26}p_{34}.
\]
They are independent: the monomials
$x_1x_4x_5$, $x_1x_2x_5$, $x_1x_3x_4$, and $x_1x_2x_3$ occur, respectively,
in only one of the four displayed polynomials.  Hence
$\dim P_3(H_5^\ast)=4$.  Therefore
\[
 \Hilb(\mathcal D(G_5^\ast),t)=1+5t+9t^2+3t^3,
 \qquad
 \Hilb(\mathcal D(H_5^\ast),t)=1+5t+9t^2+4t^3.
\]
\end{proof}

\begin{remark}[Minimality]\label[remark]{rmk:matching-counterexample-minimality}
The examples in parts (a), (b), and (c) of
\cref{prop:matching-polynomial-counterexample} are minimal in their respective
classes: among pairs of graphs on the same number of vertices with equal
matching numbers but different Dunkl Hilbert series, the smallest possible
number of vertices is $4$ in general, $5$ for connected graphs, and $6$ for
$2$-connected graphs.
\end{remark}

\subsection{Problems}

\begin{problem}
Describe the kernel
\[
 I_G=\ker\left(
 \kk[z_v:v\in V]/(\sum_{v\in C}z_v:C\text{ a component})
 \longrightarrow\mathcal D(G)\right).
\]
It contains $z_v^2$ for every $v$, but in general has additional
graph-dependent relations.
\end{problem}

\begin{problem}
Find deletion recurrences for $\dim P_k(G)$ (say, for $\kk = \mathbb Q$).
Since matchings depend on
vertex-disjointness, vertex deletion and deletion of the closed edge
neighborhood are more natural here than ordinary matroid deletion--contraction.
\end{problem}

The results on saturation proved above leave a gap for $k\geq3$:
\[
 k\text{-linked}\Longrightarrow k\text{-saturated}
 \Longrightarrow k\text{-connected}.
\]
That is, we have a necessary combinatorial condition and a sufficient
combinatorial condition for a graph to be $k$-saturated, but no condition
that is both at once (except for $k\leq 2$).

\begin{problem}
Is there such a condition, or is it too much to ask for?
\end{problem}

\section{Log-concavity conjectures}\label{sec:log-concavity}

Throughout this section, assume that $\kk$ has characteristic zero.
As Remark~\ref{rmk:characteristic-dependence} shows, the Hilbert function
can change in positive characteristic.
For a finite simple graph $G=(V,E)$, write
\[
 H_G(t):=\Hilb(\mathcal D(G),t)
       =\sum_{k=0}^{\nu(G)}h_k(G)t^k,
 \qquad r(G):=|V|-c(G),
\]
where $\nu(G)$ is the maximum size of a matching of $G$ and $c(G)$ is
its number of connected components, including isolated vertices.
We set $h_k(G)=0$ for $k>\nu(G)$.
By \cref{thm:product-realization}, we have $h_k(G)=\dim P_k(G)$
for every $k\geq0$.
In particular, $h_0(G)=1$, $h_1(G)=r(G)$, and $h_k(G)>0$ for
$0\leq k\leq\nu(G)$: a matching of size $k$ gives a nonzero polynomial
$p_M\in P_k(G)$.

\begin{conjecture}[Log-concavity]\label[conjecture]{conj:dunkl-log-concavity}
For every finite simple graph $G$, the Hilbert function of $\mathcal D(G)$
is log-concave.  That is,
\[
 h_k(G)^2\geq h_{k-1}(G)h_{k+1}(G)
 \qquad\text{for }1\leq k<\nu(G).
\]
\end{conjecture}

Since the Hilbert function has no internal zeros, this conjecture would
also imply its unimodality.  A stronger version uses $r(G)$, the dimension
of the degree-one part, as the order of normalization.

\begin{conjecture}[Ultra-log-concavity]\label[conjecture]{conj:dunkl-ultra-log-concavity}
For every finite simple graph $G$, the sequence
\[
 \left(\frac{h_k(G)}{\binom{r(G)}{k}}\right)_{k=0}^{\nu(G)}
\]
is log-concave.  Equivalently, writing $r=r(G)$, we have
\[
 k(r-k)h_k(G)^2
 \geq (k+1)(r-k+1)h_{k-1}(G)h_{k+1}(G)
 \qquad\text{for }1\leq k<\nu(G).
\]
\end{conjecture}

The parameter $r(G)$ has an algebraic explanation.  Omitting one vertex
from each connected component leaves $r(G)$ Dunkl elements that still
generate $\mathcal D(G)$, and each of these elements has square zero by
\eqref{eq:elementary-relations}.  Thus $\mathcal D(G)$ is a graded quotient of
\[
 \kk[z_1,\ldots,z_r]/(z_1^2,\ldots,z_r^2).
\]
In particular, $h_k(G)\leq\binom rk$.  In fact, passing to a monomial
initial ideal and counting inclusions between its squarefree standard
monomials gives
\[
 (k+1)h_{k+1}(G)\leq(r-k)h_k(G)
 \qquad\text{for }0\leq k<r.
\]
Hence the normalized sequence in
Conjecture~\ref{conj:dunkl-ultra-log-concavity} is already nonincreasing;
the conjecture asks for its log-concavity as well.

The choice of normalization matters.  By \cref{thm:complete-specht},
\[
 H_{K_5}(t)=1+4t+5t^2.
\]
Normalizing its coefficients by $\binom{\nu(K_5)}k=\binom2k$ gives
$(1,2,5)$, which is not log-concave, whereas normalization by
$\binom{r(K_5)}k=\binom4k$ gives $(1,1,5/6)$.
The same example also rules out real-rootedness of $H_G(t)$ in general.
Exact rational rank computations verify
Conjectures~\ref{conj:dunkl-log-concavity} and~\ref{conj:dunkl-ultra-log-concavity}
for all simple graphs on at most seven vertices, up to isomorphism.

There is also a representation-theoretic strengthening of
Conjecture~\ref{conj:dunkl-log-concavity}.  The group
$\Gamma=\operatorname{Aut}(G)$ acts on $\kk[x_v:v\in V]$ by
$\sigma(x_v)=x_{\sigma(v)}$.  This action preserves $P_k(G)$: it sends
each matching product $p_M$ to $\pm p_{\sigma(M)}$.
All tensor products below are over $\kk$, with the diagonal group action.

\begin{conjecture}[Strong equivariant log-concavity]
\label[conjecture]{conj:dunkl-equivariant-log-concavity}
For every finite simple graph $G$ and all integers
$1\leq i\leq j<\nu(G)$, there exists an
$\operatorname{Aut}(G)$-equivariant injection
\[
 P_{i-1}(G)\otimes P_{j+1}(G)
 \lhook\joinrel\longrightarrow
 P_i(G)\otimes P_j(G).
\]
\end{conjecture}

Taking dimensions and setting $i=j$ recovers
Conjecture~\ref{conj:dunkl-log-concavity}.  Li's theorem~\cite{Li} proves the
analogous statement for the permutation representations on matchings.
It does not immediately imply the conjecture above: the matching products
have orientation signs under graph automorphisms and can satisfy linear
relations, and an injection before passing to a quotient need not descend
to an injection afterwards.

For $G=K_n$, \cref{thm:complete-specht} identifies the conjectured injection
with an $S_n$-equivariant injection
\[
 S^{(n-i+1,i-1)}\otimes S^{(n-j-1,j+1)}
 \lhook\joinrel\longrightarrow
 S^{(n-i,i)}\otimes S^{(n-j,j)}
 \qquad
 \left(1\leq i\leq j<\left\lfloor\frac n2\right\rfloor\right).
\]
Here a zero second part is omitted.  Equivalently, the symmetric function
\[
 s_{(n-i,i)}*s_{(n-j,j)}
 -s_{(n-i+1,i-1)}*s_{(n-j-1,j+1)}
\]
should be Schur-positive, with $*$ denoting the Kronecker (internal)
product.  For fixed $i,j$, this specialization holds for all sufficiently
large $n$, by Gui's inequalities for reduced Kronecker coefficients
indexed by one-row partitions~\cite[Theorem~5.4]{GuiReducedKronecker}.
The conjecture asks for every admissible $n$.

\appendix

\section{Cut quotients and zonotopal algebras: a comparison}\label{sec:cut-quotients}

In this section, we leave the matching algebra $\mathcal M(G)$
behind and consider some other quotients of
$\kk[u_e : e \in E]$.
We shall define analogues of the Dunkl matching subalgebra in
these quotients, and state their Hilbert series.
Almost all the work here has been done in \cite{PS}
(sometimes implicitly), so we keep our proofs short.

Assume that the field $\kk$ has characteristic $0$.
We shall work in the \emph{square-zero edge
algebra}
\[
 \widehat\Phi_G:=\kk[u_e:e\in E]/(u_e^2:e\in E),
\]
which is larger than $\mathcal M(G)$ because it only
quotients out squares of edges but leaves
products of distinct incident edges intact.
We define the same
signed incidence sums $\theta_v$ as before, but now inside this algebra.
They generate a subalgebra
\[
 \mathcal C_G^{\mathrm{ex}}:=\kk[\theta_v:v\in V]\subseteq\widehat\Phi_G,
\]
which is known as
the external graphical zonotopal algebra (the forest version of the
Postnikov--Shapiro algebra); see \cite[\S11]{PS}, as well as the zonotopal
framework in \cite{AP,HR}.

We shall next see what happens when we impose further
monomial relations coming from edge cuts.
For a nonempty proper subset $S\subset V$, let
\[
 \delta(S):=\{e\in E:|e\cap S|=1\},
 \qquad
 d_S:=|\delta(S)|,
 \qquad
 \eta_S:=\prod_{e\in\delta(S)}u_e.
\]
Because edges internal to $S$ cancel in the incidence sum,
we have
\begin{equation}\label{eq:cut-sum}
 \Theta_S:=\sum_{v\in S}\theta_v
 =\sum_{e\in\delta(S)}\varepsilon_{S,e}u_e,
 \qquad \varepsilon_{S,e}\in\{\pm1\}.
\end{equation}
Since the $u_e$ are square-zero, the only surviving terms in the
$d_S$th power $\Theta_S^{d_S}$
use every cut edge $e \in \delta(S)$ exactly once.  Therefore
\begin{equation}\label{eq:cut-power}
 \Theta_S^{d_S}
 =d_S!\left(\prod_{e\in\delta(S)}\varepsilon_{S,e}\right)\eta_S
 =\pm d_S!\,\eta_S.
\end{equation}
In characteristic zero this shows, in particular, that every cut monomial
$\eta_S$ already belongs to $\mathcal C_G^{\mathrm{ex}}$.

The relation with the usual external and central power ideals makes the next
quotient transparent.  Put
\[
 R_V:=\kk[z_v:v\in V]/\left(\sum_{v\in V}z_v\right),
 \qquad
 \ell_S:=\sum_{v\in S}z_v.
\]
The standard graphical power-ideal presentations
(\cite[Theorem 11.1 and Corollary 10.5]{PS}) are
\begin{align}
 \mathcal C_G^{\mathrm{ex}}
 &\cong R_V/(\ell_S^{d_S+1}:\varnothing\ne S\subsetneq V),
 \label{eq:external-power-presentation}\\
 \mathcal C_G^{\mathrm{central}}
 &\cong R_V/(\ell_S^{d_S}:\varnothing\ne S\subsetneq V).
 \label{eq:central-power-presentation}
\end{align}
Here $z_v$ maps to $\theta_v$.  These presentations are the graphical cases
of the external and central zonotopal algebras; see
\cite[Secs.~9 and~11]{PS} and \cite{AP,HR}.

\begin{proposition}[The cut quotient inside the incidence-sum algebra]
\label[proposition]{prop:central-quotient}
For a connected graph $G = (V,E)$ in characteristic zero,
\[
 \mathcal C_G^{\mathrm{ex}}/
 (\eta_S:\varnothing\ne S\subsetneq V)
 \cong\mathcal C_G^{\mathrm{central}}.
\]
Consequently, if $m=|E|$ and $r=|V|-1$, then
\[
 \Hilb\left(
 \mathcal C_G^{\mathrm{ex}}/(\eta_S),t\right)
 =\Hilb\left( \mathcal C_G^{\mathrm{central}} ,t\right)
 =t^{m-r}T_G(1,t^{-1}),
\]
where $T_G$ denotes the Tutte polynomial of $G$
(see, e.g., \cite{Goodall}).
\end{proposition}

\begin{proof}
Under the presentation \eqref{eq:external-power-presentation}, the element
$\ell_S^{d_S}$ maps, by \eqref{eq:cut-power}, to a nonzero scalar multiple of
$\eta_S$.  Thus quotienting $\mathcal C_G^{\mathrm{ex}}$ by all $\eta_S$
adds exactly the relations $\ell_S^{d_S}=0$.  The resulting presentation is
\eqref{eq:central-power-presentation}.  The Hilbert-series formula is the
standard central zonotopal specialization of the Tutte polynomial
\cite{AP,HR}; e.g., combine \cite[Theorem 3.3 or Theorem 9.1]{PS}
with \cite[Theorem 3.12]{Goodall}.
\end{proof}

There is a second quotient by the same cut monomials $\eta_S$,
obtained by imposing
them on the \emph{whole} square-zero edge algebra rather than only on its
incidence-sum subalgebra.  This is larger, and it has an elementary basis.
Let
\[
 I_{\mathrm{cut}}:=(\eta_S:\varnothing\ne S\subsetneq V)
 \subseteq\widehat\Phi_G.
\]
For $F\subseteq E$, write $u_F:=\prod_{e\in F}u_e$.

\begin{proposition}[The ambient cut quotient]\label[proposition]{prop:reliability-quotient}
For connected $G$, the set
\[
 \{u_F:G\setminus F\text{ is connected}\}
\]
is a basis of $\widehat\Phi_G/I_{\mathrm{cut}}$.  Hence
\begin{align*}
 \Hilb(\widehat\Phi_G/I_{\mathrm{cut}},t)
 &=\sum_{\substack{H\subseteq G;\\H\text{ connected spanning}}}
 t^{m-|H|}\\
 &=t^{m-r}T_G(1,1+t^{-1}).
\end{align*}
\end{proposition}

\begin{proof}
The squarefree monomials $u_F$ form a basis of $\widehat\Phi_G$, and
$I_{\mathrm{cut}}$ is a monomial ideal.  Thus $u_F$ survives in the quotient
exactly when it is not divisible by any $\eta_S$, that is, exactly when
$\delta(S)\nsubseteq F$ for every nonempty proper $S\subset V$.  But a deletion
set $F$ contains a whole cut $\delta(S)$ if and only if $G\setminus F$ has no
edge from $S$ to $V\setminus S$, equivalently if and only if $G\setminus F$
is disconnected.  This proves the basis statement and the first Hilbert-series
formula.

For the second formula, use the corank--nullity expansion
\[
 T_G(1,y)=
 \sum_{\substack{H\subseteq G;\\H\text{ connected spanning}}}
 (y-1)^{|H|-r}
\]
(see, e.g., \cite[Proposition 3.10]{Goodall})
and substitute $y=1+t^{-1}$.
\end{proof}

Thus the same cut relations produce two different but closely related
objects.  They fit into the diagram
\begin{align}
 \begin{array}{ccc}
 \mathcal C_G^{\mathrm{ex}}&\subseteq&\widehat\Phi_G\\[2mm]
 \downarrow&&\downarrow\\[2mm]
 \mathcal C_G^{\mathrm{central}}&\subseteq&
 \widehat\Phi_G/I_{\mathrm{cut}}
 \end{array}
\end{align}
(whose bottom row is injective because of \cite[Corollary 10.5]{PS}%
\footnote{Indeed, the algebra $\mathcal B_G$ in \cite{PS} is
our $\mathcal C_G^{\mathrm{central}}$, while the algebra
$\Phi_G$ in \cite{PS} is our $\widehat\Phi_G/I_{\mathrm{cut}}$.
Thus, by identifying $\mathcal B_G$ with the subalgebra
$\mathcal C_G$ of $\Phi_G$, \cite[Corollary 10.5]{PS} shows
that $\mathcal C_G^{\mathrm{central}}$ embeds into
$\Phi_G = \widehat\Phi_G/I_{\mathrm{cut}}$.}).
The lower-left algebra has total dimension $T_G(1,1)$, the number of spanning
trees (by Proposition~\ref{prop:central-quotient}).
The lower-right algebra has total dimension $T_G(1,2)$, the number of
connected spanning subgraphs (by Proposition~\ref{prop:reliability-quotient});
its grading records the number of deleted edges.  More precisely, if
\[
 R_G(p)=\sum_{\substack{H\subseteq G;\\H\text{ connected spanning}}}
 p^{|H|}(1-p)^{m-|H|}
\]
is the all-terminal reliability polynomial, then
\[
 \Hilb(\widehat\Phi_G/I_{\mathrm{cut}},t)
 =(1+t)^mR_G((1+t)^{-1}).
\]

\end{document}